\documentclass[a4paper]{amsart}

\usepackage{graphicx} 
\usepackage{amsmath,amssymb,amsthm}
\usepackage{eulervm}
\usepackage{mathrsfs}
\usepackage{array}
\usepackage{comment}
\usepackage{color}
\usepackage{graphics}
\usepackage{tikz}
\usepackage{marvosym}
\usepackage{hyperref}
\usepackage[shortlabels]{enumitem}
\setlist{font=\normalfont}\usepackage{array}

\usepackage{tcolorbox}
\tcbuselibrary{theorems}

\theoremstyle{plain}
\newtheorem{theorem}{Theorem}

\newtheorem{corollary}[theorem]{Corollary}
\newtheorem{lemma}[theorem]{Lemma}
\newtheorem{proposition}[theorem]{Proposition}
\newtheorem{fact}[theorem]{Fact}
\theoremstyle{definition}
\newtheorem{remark}[theorem]{Remark}
\newtheorem{definition}[theorem]{Definition}
\newtheorem*{notation*}{Notation}
\newtheorem{example}[theorem]{Example}

\newtheorem{problem}[theorem]{Problem}

\renewcommand{\leq}{\leqslant}
\renewcommand{\geq}{\geqslant}
\renewcommand{\nleq}{\nleqslant}

\renewcommand{\phi}{\varphi}
\renewcommand{\rho}{\varrho}
\renewcommand{\epsilon}{\varepsilon}

\newcommand{\N}{\mathbb{N}} 

\newcommand{\FL}[1]{\operatorname{FL}_{#1}}         

\newcommand{\Con}[1]{\operatorname{Con}{#1}}        
\newcommand{\Eq}[1]{\operatorname{Eq}({#1})}        
\newcommand{\Set}[2]{\left\{{#1} : {#2}\right\}}      
  \newcommand{\PSet}[1]{\mathscr{P}(#1)}         

\newcommand{\tworepr}{$2$-rep\-re\-sent\-able}
\newcommand{\tworepry}{$2$-rep\-re\-sent\-abil\-i\-ty}

\newcommand{\twogood}{$2$-con-embedding}
\DeclareMathOperator{\spec}{Sp}

\newcommand{\av}{\underline{a}}

\newcommand{\ev}{\underline{e}}
\newcommand{\uv}{\underline{u}}
\newcommand{\vv}{\underline{v}}
\newcommand{\alg}{\mathbb{A}}
\newcommand{\abs}[1]{\left\lvert #1 \right\rvert}
\newcommand{\compl}[1]{\overline{#1}}

\newcolumntype{L}{>{$}l<{$}}  \newcolumntype{R}{>{$}r<{$}}  \newcolumntype{C}{>{$}c<{$}}

\newcounter{LCard}
\newcounter{LType}

\newcommand{\RLatticeItem}[2]{%
\ifnum\value{LCard}<4
  {{\bf\theLCard.}$^{(\text{#1})}$.}~{#2}\medskip\addtocounter{LCard}{1}
\else
  {{\bf\theLCard/\theLType.}$^{(\text{#1})}$}~{#2}\medskip\addtocounter{LType}{1} 
\fi}
\newcommand{\NRLatticeItem}[2]{%
  \textcolor{red}{{\bf\theLCard/\theLType.}$^{(\text{#1})}$}~#2\medskip\addtocounter{LType}{1}}

\begin{document} 

\title{Congruence lattices of $2$-valued unary algebras}

\author[M. Dormán]{Miklós Dormán}
\address{Bolyai Institute\\
     Aradi vértanúk tere 1.\\
     H-6720 Szeged\\
     Hungary}
\email{dorman@math.u-szeged.hu}

\author[K.~K\'atai-Urb\'an]{Kamilla~K\'atai-Urb\'an}
\address{Bolyai Institute\\
     Aradi vértanúk tere 1.\\
     H-6720 Szeged\\
     Hungary}
\email{katai@math.u-szeged.hu}

\author[E.~K.~Horv\'ath]{Eszter~K.~Horv\'ath}
\address{Bolyai Institute\\
     Aradi vértanúk tere 1.\\
     H-6720 Szeged\\
     Hungary}
\email{horeszt@math.u-szeged.hu}

\author[T.~Waldhauser]{Tam\'as~Waldhauser}
\address{Bolyai Institute\\
     Aradi vértanúk tere 1.\\
     H-6720 Szeged\\
     Hungary}
\email{twaldha@math.u-szeged.hu}


\dedicatory{Dedicated to Ivan Chajda on the occasion of his 80th birthday.}

\thanks{The research of the authors was partially supported by the National Research, Development and Innovation Office of Hungary under grant ADVANCED 153383.} 

\begin{abstract}
We study congruence lattices of finite algebras whose basic operations are unary and have exactly two values.
We prove that the class of finite lattices representable in this way is closed under taking filters, glued sums and ordinal sums, and direct products with finite distributive lattices, but it is not closed under taking ideals or arbitrary direct products. 
A classification of lattices with at most seven elements with respect to representability is also presented.

\end{abstract}

\maketitle


\section{Introduction}

The representation of lattices as congruence lattices is a classical problem of universal algebra.
Gr\"atzer and Schmidt \cite{GS63} proved that every algebraic lattice is isomorphic to the congruence lattice of an algebra.
Their theorem does not guarantee that the algebra is finite when the lattice is finite.
The abstract finite congruence lattice representation problem asks whether every finite lattice is isomorphic to the congruence lattice of a finite algebra.
For this problem, it is enough to consider unary algebras: replacing the basic operations of an algebra by all its unary polynomial operations does not change its congruence lattice \cite[Chapter~II, \S5, Exercise~7]{BS}.

There is also a concrete version of the representation problem.
Given a finite set $A$ and a $(0,1)$-sublattice $L$ of the lattice $\Eq{A}$ of equivalence relations on $A$, one asks whether $L$ is the congruence lattice of an algebra on $A$.
The preservation relation between unary operations and equivalence relations gives a Galois connection whose closed sets on the relational side are precisely the concrete congruence lattices.
This implies immediately that a $(0,1)$-sublattice of $\Eq{A}$ is a congruence lattice of some algebra on $A$ if and only if it is closed under primitive positive definability \cite[Corollary 2.2]{Snow00}.
This was observed earlier by Werner \cite{Wer76} in terms of graphical compositions.

Jakub\'\i{}kov\'a-Studenovsk\'a, P\"oschel and Radeleczki \cite{JSPR18} studied the lattice formed by all congruence lattices on a fixed finite set.
Subsequently, they described its maximal proper members in detail \cite{JSPR21}.
That paper also gives a description of the congruences of a monounary algebra as a union of intervals of the partition lattice.
Related Galois connections involving generalized quasiorders are studied in \cite{JSPR}.

In this paper we study the Galois connection induced by the preservation relation between equivalence relations and unary operations with exactly two values.
Constant operations preserve every equivalence relation, so two-element ranges give the first nontrivial case when restricting the range sizes of unary operations.
We write $T_2(A) = \big\{ f\colon A \to A : \abs{f(A)} = 2 \big\}$ and call an algebra $\alg=(A;F)$ with $F\subseteq T_2(A)$ a $2$-valued unary algebra.
The Galois closed classes of relations are then exactly the congruence lattices of $2$-valued unary algebras.

Our main result in Section~\ref{sect_closure} gives an explicit description of the Galois closure of a set of equivalence relations on a finite set $A$ as a union of certain intervals of $\Eq{A}$ determined by lattice terms in disjunctive and conjunctive normal forms (see Theorem~\ref{thm_int}).
This yields a solution to the concrete representation problem when restricted to $2$-valued unary algebras, which was already obtained by Bandelt \cite[Corollary~3]{Bandelt}.

Turning our attention to the abstract representation problem, we introduce the following terminology.

\begin{definition}\label{def_goodfi}
	We say that a finite lattice $L$ is \emph{\tworepr{}} if there is a finite set $A$ and an embedding $\varphi\colon L\hookrightarrow \Eq{A}$ such that the image of $\phi$ is the congruence lattice of an algebra $\alg=(A;F)$ with $F \subseteq T_2(A)$.
	In this case we say that $\phi$ is a \emph{\twogood{}}.
	Note that all \twogood{}s are $(0,1)$-embeddings.
\end{definition}

Every finite partition lattice is \tworepr{}, since $\Eq{A}=\Con{(A;\emptyset)}$.
Finite distributive lattices form another basic class of examples.
Quackenbush and Wolk \cite{QW71} proved that every distributive $(0,1)$-sublattice of $\Eq{A}$ is the congruence lattice of an algebra on $A$, and their construction uses only $2$-valued unary operations.
Pudl\'ak \cite{Pu77} proved that finite strongly representable lattices, i.e., lattices for which every $(0,1)$-embedding into a partition lattice gives a congruence lattice, must be distributive.
However, the characterization of \tworepr{} nondistributive lattices remains an open problem.

For unrestricted finite representability, closure under intervals, direct products, glued sums and ordinal sums is known; see \cite[Lemmas~3.4--3.6 and Corollary~3.7]{Snow00}.
We prove in Section~\ref{sect_constructions} that \tworepry{} is preserved by taking filters, glued sums and ordinal sums, and by taking the direct product with a finite distributive lattice.
On the other hand, closure under arbitrary direct products fails: for example, $M_3\times M_3$ is not \tworepr{}.
Nevertheless, if $L_1$ and $L_2$ are \tworepr{}, then $(L_1\times L_2)\oplus\mathbf{1}$ (i.e., the lattice obtained from $L_1\times L_2$ by adjoining an extra top element) is \tworepr{}.
Taking an ideal need not preserve \tworepry{}, but adding an extra top element to an ideal of a \tworepr{} lattice yields again a \tworepr{} lattice.
These results provide constructions of \tworepr{} lattices as well as obstructions to such representations.

In Section~\ref{sect_conclusion} we formulate some open problems related to \tworepry{}, and we also consider the possible sizes of $2$-valued unary algebras with a given congruence lattice.

We present the results of computer studies of small algebras and small lattices in the Appendix.
More precisely, in Section~\ref{sect_four-element-sets} we list the congruence lattices of $2$-valued unary algebras of at most four elements.
Section~\ref{sect_seven-element-lattices} contains a classification of lattices of size at most seven with respect to \tworepry{}.
For \tworepr{} lattices, a suitable $2$-valued unary algebra was found by computer; for the remaining lattices we give appropriate lattice terms in disjunctive and conjunctive normal forms that contradict \tworepry{} according to Theorem~\ref{thm_int}.

\section{Preliminaries and notation}

\subsection{General notational conventions}

The set of natural numbers is denoted by $\N$ (zero is not included).
For $n \in \N$, we use the notation $[n] = \{1,2,\dots,n\}$, and $\PSet{[n]}$ stands for the power set of $[n]$.
Letters $a,b,c$ refer to elements of a set $A$, and the letters $u,v,w$ refer to elements of a lattice $L$.
Underlined letters denote tuples; if not indicated otherwise, tuples are of length $n$, and their entries are indexed by $1,\dots,n$, such as in $\av=(a_1,\dots,a_n) \in A^n$ and $\uv=(u_1,\dots,u_n) \in L^n$.

\subsection{Lattices}

We consider finite lattices; these lattices are always bounded: there is a least (bottom) element $0_L$ and a greatest (top) element $1_L$.
Sublattices containing these boundary elements are called $(0,1)$-sublattices, and, similarly, homomorphisms preserving the boundary elements are referred to as $(0,1)$-homomorphisms.
The embeddings used to represent lattices as congruence lattices will always preserve $0$ and $1$.
The elements covering $0_L$ (covered by $1_L$) are called the atoms (coatoms) of $L$.
The interval determined by the elements $u,v \in L$ is the set $[u;v]_L = \Set{w \in L}{u \leq w \leq v} \subseteq L$; of course, this interval is nonempty only if $u \leq v$.
Intervals of the form $[0_L;u]_L$ and $[u;1_L]_L$ are called principal ideals and principal filters, respectively.
The (isomorphism types of) one- and two-element lattices are denoted by $\mathbf{1}$ and $\mathbf{2}$.

\begin{definition}\label{def_ordsum}
	The \emph{ordinal sum} $L_1\oplus L_2$ of two lattices is the lattice obtained by putting $L_2$ on top of $L_1$. More precisely, the underlying set of the ordinal sum is the disjoint union $L_1\mathop{\dot\cup} L_2$; the order is defined by keeping the original orders within $L_1$ and $L_2$ and declaring $u<v$ whenever $u\in L_1$ and $v\in L_2$.
	The \emph{glued sum} $L_1\oplus_w L_2$ of two bounded lattices is defined in a similar manner, but here the greatest element of $L_1$ is identified with the least element of $L_2$, and $w$ denotes this identified element.
	Figure~\ref{fig_ordinal-sum} illustrates these two constructions.
\end{definition}

\begin{figure}[ht]
	\centering
	\begin{tikzpicture}[every node/.style={font=\small}]
	
	\tikzset{
	    latticeoutline/.style={thick},
	    extremum/.style={circle, fill=black, inner sep=1.2pt}
	}
	
	\begin{scope}[xshift=0cm]
	
	    \draw[latticeoutline]
	        (0,0) .. controls (-0.95,0.7) and (-0.95,1.5) .. (0,2.2)
	              .. controls ( 0.95,1.5) and ( 0.95,0.7) .. (0,0);
	    \node at (0,1.1) {$L_1$};
	
	    \draw[latticeoutline]
	        (0,2.9) .. controls (-0.95,3.6) and (-0.95,4.4) .. (0,5.1)
	                .. controls ( 0.95,4.4) and ( 0.95,3.6) .. (0,2.9);
	    \node at (0,4.0) {$L_2$};
	
	    \draw[latticeoutline] (0,2.2) -- (0,2.9);
	
	    \node[extremum, label=left:{$0_{L_1}$}] at (0,0) {};
	    \node[extremum, label=left:{$1_{L_1}$}] at (0,2.2) {};
	    \node[extremum, label=left:{$0_{L_2}$}] at (0,2.9) {};
	    \node[extremum, label=left:{$1_{L_2}$}] at (0,5.1) {};
	
	    \node at (0,-0.9) {$L_1\oplus L_2$};
	
	\end{scope}
	
	\begin{scope}[xshift=7cm]
	
	    \draw[latticeoutline]
	        (0,0) .. controls (-0.95,0.7) and (-0.95,1.5) .. (0,2.2)
	              .. controls ( 0.95,1.5) and ( 0.95,0.7) .. (0,0);
	    \node at (0,1.1) {$L_1$};
	
	    \draw[latticeoutline]
	        (0,2.2) .. controls (-0.95,2.9) and (-0.95,3.7) .. (0,4.4)
	                .. controls ( 0.95,3.7) and ( 0.95,2.9) .. (0,2.2);
	    \node at (0,3.3) {$L_2$};
	
	    \node[extremum, label=left:{$0_{L_1}$}] at (0,0) {};
	    \node[extremum, label=right:{$w=1_{L_1}=0_{L_2}$}] at (0,2.2) {};
	    \node[extremum, label=left:{$1_{L_2}$}] at (0,4.4) {};
	
	    \node at (0,-0.9) {$L_1\oplus_w L_2$};
	
	\end{scope}
	
	\end{tikzpicture}
	\caption{The ordinal sum $L_1\oplus L_2$ and the glued sum $L_1\oplus_w L_2$.}
	\label{fig_ordinal-sum}
\end{figure}
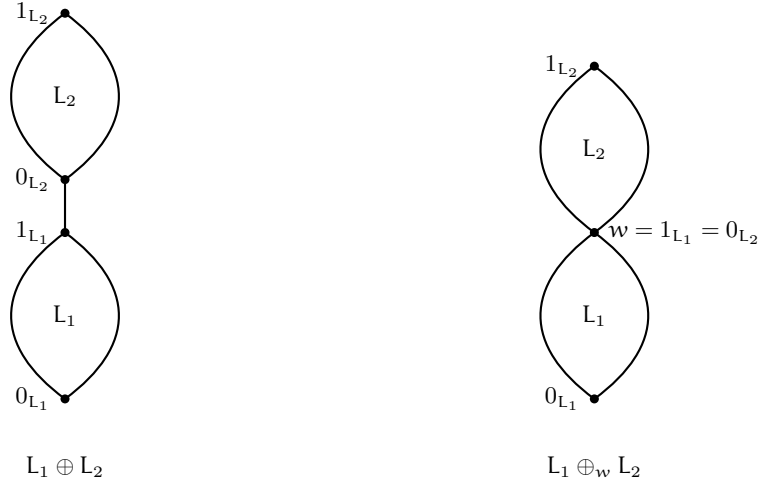

\begin{remark}\label{rem_oplus}
	The following connection holds between the ordinal sum and the glued sum: $(L_1\oplus_v\mathbf{2} )\oplus_w L_2\cong L_1\oplus L_2$, where $v$ and $w$ denote the identified elements.
\end{remark}

\subsection{Disjunctive and conjunctive normal forms}\label{subsect_DNF}

Let $\FL{n}$ denote the free bounded lattice on $n$ free generators; in other words, $\FL{n}$ is the set of all lattice terms in the variables $x_1,\dots,x_n$.
We allow the empty meet $\bigwedge \emptyset = 1$ and the empty join $\bigvee \emptyset = 0$; this will not cause any problems, as we consider finite (hence bounded) lattices.
\begin{definition}[{\cite[Chapter~V, \S4.1, Theorem~404]{Gratzer11}}]
	A lattice term is said to be in \emph{disjunctive normal form} (DNF) if it is of the following form for some $\mathcal{M}\subseteq \PSet{[n]}$:
	\[
		t^-_{\mathcal{M}} = \bigvee_{M\in\mathcal{M}}\bigwedge_{i\in M}x_i \in \FL{n}.
	\]
	Similarly, the \emph{conjunctive normal form} (CNF) corresponding to $\mathcal{N}\subseteq \PSet{[n]}$ is the lattice term
	\[
		t^+_{\mathcal{N}} = \bigwedge_{N\in\mathcal{N}}\bigvee_{i\in N}x_i \in \FL{n}.
	\]
\end{definition}
Different set systems $\mathcal{M}$ may induce disjunctive normal forms that are equivalent over every bounded lattice; however, if we assume that $\mathcal{M}$ is an upset (i.e., $M \in \mathcal{M}$ and $M \subseteq M' \subseteq [n]$ imply $M' \in \mathcal{M}$), then we can recover $\mathcal{M}$ from the term operation by evaluating it at special tuples.
Similarly, there is a bijective correspondence between upsets $\mathcal{N}\subseteq \PSet{[n]}$ and conjunctive normal forms $t^+_{\mathcal{N}}$.
To see this, let us define the tuple $\ev_{I}\in L^n$ for a lattice $L$ and $I\subseteq [n]$ by
\[
  e_i=\begin{cases}
    1_L, & \text{ if } i\in I,\\
    0_L, & \text{ if } i\notin I.
  \end{cases}
\]
If $\mathcal{M}$ and $\mathcal{N}$ are upsets, then
\begin{equation}\label{equ_t(e_I)}
t^-_\mathcal{M}(\ev_I)=\begin{cases}
  1_L, &\text{ if } \ I\in \mathcal{M},\\
  0_L, &\text{ if } \ I\notin \mathcal{M},	
\end{cases}
\qquad
t^+_\mathcal{N}(\ev_I)=\begin{cases}
  1_L, &\text{ if } \ \compl{I}\notin \mathcal{N},\\
  0_L, &\text{ if } \ \compl{I}\in \mathcal{N},
\end{cases}
\end{equation}
where $\compl{I} = [n] \setminus I$ denotes the complement of $I$.

Over distributive lattices, every term is equivalent to a DNF, and we shall denote the DNF corresponding to $t \in \FL{n}$ by $t^-$ or $t^-_{\mathcal{M}}$; the latter notation also indicates the underlying upset $\mathcal{M}\subseteq \PSet{[n]}$.
Similarly, for every $t \in \FL{n}$, there is a unique upset $\mathcal{N}\subseteq \PSet{[n]}$ such that the identity $t \approx t^+_{\mathcal{N}}$ holds in every bounded distributive lattice.
The notation is motivated by the fact that $t^-$ and $t^+$ are the least, respectively, greatest elements of the free bounded lattice $\FL{n}$ that are equivalent to $t$ over the variety of bounded distributive lattices.
In the next proposition we prove only the weaker statement $t^- \leq t^+$, as we only need this in the sequel, and we also describe the connection between the upsets $\mathcal{M}$ and $\mathcal{N}$ corresponding to the DNF $t^-_{\mathcal{M}}$ and to the CNF $t^+_{\mathcal{N}}$ of the same term $t \in \FL{n}$.

\begin{proposition}\label{prop_MN}
	For all upsets $\mathcal{M}, \mathcal{N}\subseteq \PSet{[n]}$, the following hold.
	\begin{enumerate}[(i)]
		\item\label{item_MN-identity} The identity $t^-_{\mathcal{M}}\approx t^+_{\mathcal{N}}$ holds in the variety of bounded distributive lattices if and only if $\mathcal{N}=\Set{ \compl{I} }{ I\notin\mathcal{M} }$.
		\item\label{item_MN-inequality} If $\mathcal{N}=\Set{ \compl{I} }{ I\notin\mathcal{M} }$, then the inequality $t^-_{\mathcal{M}}\leq t^+_{\mathcal{N}}$ holds in the variety of all bounded lattices. 
	\end{enumerate}
\end{proposition}

\begin{proof}\ 
\begin{enumerate}[(i)]
\item Since the variety of bounded distributive lattices is generated by the two-element lattice $\mathbf{2}$, the identity $t^-_{\mathcal{M}}\approx t^+_{\mathcal{N}}$ holds in every distributive lattice if and only if $t^-_{\mathcal{M}}(\ev_I)=t^+_{\mathcal{N}}(\ev_I)$ for every $I\subseteq [n]$. According to \eqref{equ_t(e_I)}, this means that $I\in \mathcal{M}$ is equivalent to $\compl{I}\notin \mathcal{N}$.
\item The inequality 
\[
 \bigvee_{M\in\mathcal{M}}\bigwedge_{i\in M}x_i
\leq
\bigwedge_{\compl{N}\notin\mathcal{M}}\bigvee_{i\in N}x_i
\]
holds if and only if for every $M\in\mathcal{M}$ and every 
$\compl{N}\notin\mathcal{M}$, we have
$
\bigwedge_{i\in M}x_i \leq \bigvee_{i\in N}x_i.
$
Since $\mathcal{M}$ is an upset, $M\in\mathcal{M}$ and 
$\compl{N}\notin\mathcal{M}$ imply that $M\cap N\neq\emptyset$.
Hence there exists $j\in M\cap N$, and for this $j$ we have
$
\bigwedge_{i\in M}x_i \leq x_j \leq \bigvee_{i\in N}x_i.
\qedhere
$
\end{enumerate}
\end{proof}

\begin{example}\label{example_median}
For $n=3$ and $\mathcal{M} = \big\{ \{1,2\}, \, \{1,3\}, \, \{2,3\}, \, \{1,2,3\} \big\}$, the upset $\mathcal{N}$ corresponding to $\mathcal{M}$ according to Proposition~\ref{prop_MN} is $\mathcal{M}$ itself.
The DNF and CNF induced by this upset are the ternary lower and upper median terms denoted by $m^-$ and $m^+$, respectively: 
	\begin{align*}
		m^-(x,y,z)&:=t^-_{\mathcal{M}}(x,y,z)=(x\wedge y)\vee(x\wedge z)\vee(y\wedge z),\\
		m^+(x,y,z)&:=t^+_{\mathcal{N}}(x,y,z)=(x\vee y)\wedge(x\vee z)\wedge(y\vee z).
	\end{align*}
\end{example}

\subsection{Equivalence relations}

The set of equivalence relations on a set $A$ is a lattice, denoted by $\Eq{A}$; we call these lattices partition lattices (with a bit of abuse of terminology), because of the natural correspondence between equivalence relations and partitions.
The least element of $\Eq{A}$ is the equality relation, while its greatest element is the total relation $A \times A$. 
We use the simplified notation $0_A = 0_{\Eq{A}}$ and $1_A = 1_{\Eq{A}}$ for these two relations.

The equivalence class, or block, of $a \in A$ with respect to $\rho \in \Eq{A}$ is $a/\rho = \Set{b \in A}{(a,b) \in \rho}$, and the set of equivalence classes, i.e., the partition corresponding to $\rho$, is denoted by $A/\rho$.
The atoms of $\Eq{A}$ are those equivalence relations that have only one non-singleton block and this block has exactly two elements.
On the other hand, the coatoms are those equivalences that have exactly two blocks.
Finite partition lattices are atomistic and coatomistic: every element of $\Eq{A}$ is the join of the atoms below it and also the meet of the coatoms above it. 

We will need the following description of the structure of intervals in partition lattices. 
\begin{fact}\label{fact_intervals-in-Eq(A)}
	Let $\rho,\sigma \in \Eq{A}$ such that $\rho \leq \sigma$ and $A/\sigma = \{A_1,\dots,A_t\}$.
	Then we have 
	\[
		[\rho;\sigma]_{\Eq{A}} \cong \Eq{[m_1]} \times \dots \times \Eq{[m_t]},
	\] 
	where $m_i$ is the number of blocks of $\rho$ within $A_i\ (i=1,\dots,t)$.
\end{fact}

\subsection{Algebras and congruences}

The set of all transformations (unary operations) on a set $A$ (i.e., maps $f\colon A \to A$) is denoted by $T(A)$, while the set of all $2$-valued transformations is denoted by $T_2(A)$, that is, 
\[
	T_2(A)=\big\{ f\in T(A) : \abs{f(A)}=2 \big\}.
\]
By a $2$-valued unary algebra we mean an algebra $\alg=(A;F)$ with $F \subseteq T_2(A)$.

\begin{definition}
	The unary operation $f \in T(A)$ \emph{preserves} the equivalence relation $\rho \in \Eq{A}$ if $f(\rho) \subseteq \rho$, i.e., 
	\[
		\forall a,b \in A\colon\; (a,b) \in \rho \implies (f(a),f(b)) \in \rho.
	\]
	We denote this fact by $\rho \lhd f$, and we extend this notation to a set $F \subseteq T(A)$ of operations naturally: $\rho \lhd F$ means that $\rho \lhd f$ for all $f \in F$.
	We say that $\rho$ is a \emph{congruence} of the algebra $\alg=(A;F)$ if $\rho \lhd F$.
	The set of all congruences of $\alg$ is denoted by $\Con{\alg}$; note that $\Con{\alg}$ is a $(0,1)$-sublattice of $\Eq{A}$.
\end{definition}

The preservation relation $\lhd$ induces a Galois connection between ($2$-valued) transformations and equivalence relations on a given set $A$, where the Galois closed sets on the ``relational side'' are exactly the congruence lattices of ($2$-valued unary) algebras.
For this Galois connection, we use the following notation: for any $H \subseteq \Eq{A}$, let 
\begin{align*}
  H^*       &= \Set{ f\in T(A) }{ \rho\lhd f \text{ for all } \rho\in H },\\
  H^{\circledast} &= \Set{ f\in T_2(A) }{ \rho\lhd f \text{ for all } \rho\in H }
           = H^*\cap T_2(A)\subseteq H^*,\\
  H^{**}      &= \Set{ \sigma\in\Eq{A} }{ \sigma\lhd f \text{ for all } f\in H^{*} },\\
  H^{\circledast*} &= \Set{ \sigma\in\Eq{A} }{ \sigma\lhd f \text{ for all } f\in H^{\circledast} }\supseteq H^{**}.
\end{align*}
Observe that $H^{**}$, the Galois closure of $H$, is the least subset of $\Eq{A}$ containing $H$ that is the set of all congruences of some (unary) algebra $(A;F)$.
Circled asterisks indicate that we restrict our attention to $2$-valued unary operations; in particular, $H^{\circledast*}$ is the least subset of $\Eq{A}$ containing $H$ that is the set of all congruences of some $2$-valued unary algebra $(A;F)$.
We record this latter fact for reference in the following proposition.
\begin{proposition}\label{prop_Lsstar}
	Let $A$ be a finite set and $L\leq\Eq{A}$. Then $L$ is the congruence lattice of a $2$-valued unary algebra if and only if $L= L^{\circledast *}$. 
\end{proposition}

\section{A characterization of congruence lattices of $2$-valued unary algebras}\label{sect_closure}

We assign two equivalence relations on $A$ to every transformation $f\in T_2(A)$:
\[
	\alpha_f=f(A)^2\cup 0_A\quad\text{and}\quad\omega_f=\ker f.
\]
From $\abs{f(A)}=2$ it follows that $\alpha_f$ is an atom and $\omega_f$ is a coatom in $\Eq{A}$.
As the next lemma (a special case of \cite[Proposition~2.3]{JSPR21}, which already appeared in \cite{Bandelt}) shows, $\alpha_f$ and $\omega_f$ determine the equivalence relations preserved by $f$.

\begin{lemma}
	Let $f\in T_2(A)$ with range $f(A)=\{a,b\}$ and let $\rho\in\Eq{A}$. 
	Then $\rho \lhd f$ if and only if $\alpha_f \leq \rho$ or $\rho \leq \omega_f$.
\end{lemma}

\begin{proof}
	If $\alpha_f \leq \rho$, then $(a,b) \in\rho$, and in this case $(f(c),f(d)) \in \rho$ holds for all $c,d \in A$ (regardless of whether $(c,d) \in \rho$), hence $\rho\lhd f$. 
	If $(a,b) \notin \rho$, then
	\begin{align*}
		\rho \lhd f 
		&\iff \forall (c,d) \in \rho\colon (f(c),f(d)) \in \rho\\
		&\iff \forall (c,d) \in \rho\colon f(c)=f(d)\\
		&\iff \rho \leq \ker f = \omega_f.
		\qedhere
	\end{align*}  
\end{proof}

The above lemma immediately yields the following description of the Galois closure $H^{\circledast *}$, which gives a polynomial-time membership test.

\begin{lemma}\label{lemma_H**}
	Let $H \subseteq \Eq{A}$ be an arbitrary set of equivalence relations.
	For every $\sigma \in \Eq{A}$, the following three conditions are equivalent:
	\begin{enumerate}[(a)]\addtolength{\itemsep}{1mm}
		\item $\sigma\in H^{\circledast *}$;
		\item for each atom $\alpha$ and coatom $\omega$ of $\Eq{A}$, we have
		\begin{align*}
			\big(\forall \rho \in H\colon \alpha \leq \rho \text{ or } \rho \leq \omega\big)
			\ \implies\ 
			\big(\alpha \leq \sigma \text{ or } \sigma \leq \omega\big).
		\end{align*}  
		\item for each pair $(a,b) \notin \sigma$, we have
		\[
			\sigma \leq \bigvee \Set{\rho \in H}{(a,b) \notin \rho}.
		\] 
	\end{enumerate}
\end{lemma}

Our main result in this section is the next theorem, which gives a different, more lattice-theoretic description of $H^{\circledast *}$.

\begin{theorem}\label{thm_int} Let $A$ be a finite set, and let $H = \{\rho_1,\dots,\rho_n\} \subseteq \Eq{A}$ be an arbitrary nonempty set of equivalence relations \textup{(}here $n$ denotes the cardinality of $H$\textup{)}.
Then the Galois closure of $H$ can be described as follows:
\[
  \sigma \in H^{\circledast *} \iff \exists t \in \FL{n}\colon \sigma\in\big[t^-(\rho_1,\dots,\rho_n);t^+(\rho_1,\dots,\rho_n)\big].
\]
\end{theorem}

\begin{proof}
Assume that $\sigma\in H^{\circledast *}$, 
and define $\mathcal{M}=\Set{ M\subseteq [n] }{ \sigma\geq\bigwedge_{i\in M}\rho_i }$. 
Then $\mathcal{M}$ is an upset and
\[ 
  t^-_{\mathcal{M}}(\rho_1,\dots,\rho_n)=\bigvee_{M\in\mathcal{M}}\bigwedge_{i\in M}\rho_i\leq\sigma.
\]
It remains to verify the inequality 
\[
  \sigma\leq t^+_{\mathcal{N}}(\rho_1,\dots,\rho_n)=\bigwedge_{N\in\mathcal{N}}\bigvee_{i\in N}\rho_i, 
\]
where $\mathcal{N}=\Set{ \compl{I} }{ I\notin\mathcal{M} }$ (see Proposition~\ref{prop_MN}). 
Equivalently, we need to show that 
$
  \sigma\leq \bigvee_{i\in N}\rho_i
$ 
for all $N\in\mathcal{N}$.
From $\compl{N}\notin \mathcal{M}$ it follows that $\bigwedge_{i\in \compl{N}}\rho_i \nleq \sigma$ by the definition of $\mathcal{M}$. Therefore, using the fact that each element of $\Eq{A}$ is the join of the atoms below it, we see that there exists an atom $\alpha$ such that 
$ 
  \alpha\leq \bigwedge_{i\in \compl{N}}\rho_i \text{ and } \alpha\nleq\sigma.
$ 

Now let $\omega \in \Eq{A}$ be a coatom such that $\bigvee_{i\in N}\rho_i \leq \omega$. 
Then for every $i\in [n]$, either $\alpha\leq \rho_i$ (if $i\in\compl{N}$) or $\rho_i\leq\omega$ (if $i\in N$). 
Since $\sigma\in H^{\circledast *}$, it follows by Lemma~\ref{lemma_H**} that $\alpha\leq\sigma$ or $\sigma\leq\omega$. 
As $\alpha\nleq\sigma$, we conclude that $\sigma\leq\omega$.
This holds for all coatoms $\omega$ above $\bigvee_{i\in N}\rho_i$, hence $\sigma \leq \bigvee_{i\in N}\rho_i$ as claimed (here we use the fact that each element of $\Eq{A}$ is the meet of the coatoms above it).

To prove the other direction, let $t \in \FL{n}$ be an arbitrary term, let $t^-$ and $t^+$ be the corresponding DNF and CNF induced by the upsets $\mathcal{M}$ and $\mathcal{N}$, respectively.
Suppose that $\sigma\in\big[t^-(\rho_1,\dots,\rho_n);t^+(\rho_1,\dots,\rho_n)\big]$; then for every $M\in\mathcal{M}$ we have $\bigwedge_{i\in M}\rho_i \leq \sigma$ and for every $N\in \mathcal{N}$ we have $\sigma\leq \bigvee_{i\in N}\rho_i$.
Let $\alpha$ be an arbitrary atom in $\Eq{A}$, and let $I=\Set{i\in[n]}{\alpha \leq \rho_i}$.
Similarly, let $\omega$ be a coatom and let $J=\Set{i\in[n]}{\rho_i \leq \omega}$.
According to Lemma~\ref{lemma_H**}, in order to verify $\sigma \in H^{\circledast *}$, we need to prove that if $I \cup J = [n]$, then $\alpha\le \sigma$ or $\sigma\le \omega$. 
If $I\in \mathcal{M}$, then we are done, as 
$
	\alpha \leq \bigwedge_{i\in I}\rho_i \leq \sigma.
$ 
If $I\notin\mathcal{M}$, then by Proposition~\ref{prop_MN}, we have $\compl{I}\in \mathcal{N}$; furthermore, the assumption $I \cup J = [n]$ implies that $\compl{I} \subseteq J$, hence 
$
	\sigma\leq \bigvee_{i\in \compl{I}}\rho_i \leq \bigvee_{i\in J}\rho_i\leq\omega,
$
and this completes the proof.
\end{proof}

Using Proposition~\ref{prop_Lsstar} and Theorem~\ref{thm_int}, we obtain the following characterization of congruence lattices of $2$-valued unary algebras.
Note that this result is equivalent to Corollary~3 in \cite{Bandelt}.

\begin{corollary}\label{cor_intconc}
	For any finite set $A$ and any sublattice $L\leq\Eq{A}$, the following two conditions are equivalent:
	\begin{enumerate}[(a)]\addtolength{\itemsep}{1mm}
		\item there exists an algebra $\alg=(A;F)$ with $F\subseteq T_2(A)$ such that $L=\Con(\alg)$; 
		\item\label{item_interval-conditionc} for every $n \in \N$, $t \in \FL{n}$ and for all $\rho_1,\dots,\rho_n\in L$, we have
			\[
				\big[t^-(\rho_1,\dots,\rho_n);t^+(\rho_1,\dots,\rho_n)\big]_{\Eq{A}}\subseteq L.
			\]
	\end{enumerate}
\end{corollary}

If $L\leq\Eq{A}$ is a distributive $(0,1)$-sublattice, then $t^-$ and $t^+$ induce the same term operation on $L$ for every term $t$, thus the interval in condition \ref{item_interval-conditionc} of Corollary~\ref{cor_intconc} is always trivial.
This gives the following corollary, which was already proved by Quackenbush and Wolk \cite{QW71} (note that their construction gives $2$-valued unary algebras, though it is not mentioned explicitly).

\begin{corollary}\label{cor_dist}
	If $A$ is a finite set, then every distributive $(0,1)$-sublattice of $\Eq{A}$ is the congruence lattice of a $2$-valued unary algebra.
	Consequently, every finite distributive lattice is \tworepr{}.
\end{corollary}

We formulate two more consequences of Corollary~\ref{cor_intconc}: a characterization of \twogood{}s (see Definition~\ref{def_goodfi}) and a necessary condition for \tworepry{}.

\begin{corollary}\label{cor_intabst}
	For any finite lattice $L$, any finite set $A$ and any $(0,1)$-embedding $\varphi\colon L \hookrightarrow \Eq{A}$, the following three conditions are equivalent:
	\begin{enumerate}[(a)]\addtolength{\itemsep}{1mm}
		\item $\varphi$ is a \twogood{}; 
		\item\label{item_interval-condition-phi-b} for every $n \in \N$, $t \in \FL{n}$ and for all $\uv\in L^n$, we have 
			\[
				\big[\varphi(t^-(\uv));\varphi(t^+(\uv))\big]_{\Eq{A}}\subseteq \varphi(L).
			\]
		\item\label{item_interval-condition-phi-c} for every $n \in \N$, $t \in \FL{n}$ and for all $\uv\in L^n$, we have 
			\[
				\varphi\big([t^-(\uv);t^+(\uv)]_L\big) = \big[\varphi(t^-(\uv));\varphi(t^+(\uv))\big]_{\Eq{A}}.
			\]
	\end{enumerate}
\end{corollary}

\begin{corollary}\label{cor_necessary}
	If a finite lattice $L$ is \tworepr{}, then for every $n \in \N$, $t \in \FL{n}$ and for all $\uv\in L^n$, the interval $[t^-(\uv);t^+(\uv)]_L$ is isomorphic to an interval in a partition lattice.
\end{corollary}

\begin{example}\label{example_L73d}
Consider the lattice in Figure~\ref{Fig:L73d} and the terms $t^-(x,y,z)=x\vee(y\wedge z)$ and $t^+(x,y,z)=(x\vee y)\wedge (x\vee z)$. We have $t^-(u,v,w)=u$ and $t^+(u,v,w)=1$. 
However, for any (finite) set $A$, the lattice $\Eq{A}$ does not contain an interval that is a $3$-element chain (see Fact~\ref{fact_intervals-in-Eq(A)}). 
Therefore, by Corollary~\ref{cor_necessary}, the lattice is not \tworepr{}.

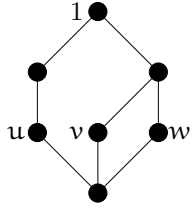
\begin{figure}[ht]
\centering
\begin{tikzpicture}[x=0.8mm,y=0.8mm, line width=0.35pt,
 lab/.style={draw=none, fill=none, below=3mm, font=\small}]

\newcommand{\vertex}[2]{\filldraw[black] (#1,#2) circle (1.2mm);}
\newcommand{\lat}[4]{%
 \begin{scope}[shift={(#1,#2)}]
  #4
 \end{scope}}

 
\lat{0}{0}{L_{7,3}^d}{
 \draw (0,0)--(-10,10)--(-10,20)--(0,30)--(10,20)--(10,10)--(0,0);
 \draw (0,0)--(0,10)--(10,20);
 
 \vertex{0}{30}
 \node at (-3.5,30) {$1$};
 \vertex{-10}{20}\vertex{10}{20}
 \vertex{-10}{10}\vertex{0}{10}\vertex{10}{10}
 \node at (-13.5,10) {$u$}; \node at (-3.5,10) {$v$}; \node at (13.5,10) {$w$};
 \vertex{0}{0}}
\end{tikzpicture}
\caption{The lattice $L_{7,3}^d$}\label{Fig:L73d}
\end{figure} 
\end{example}

The next proposition considers an extreme case of the interval condition \ref{item_interval-conditionc} of Corollary~\ref{cor_intconc}.

\begin{proposition}\label{prop_median}
	If the congruence lattice of a finite $2$-valued unary algebra $\alg$ contains a $(0,1)$-sublattice isomorphic to $M_3$, then $\Con{\alg}=\Eq{A}$, i.e., every equivalence relation is a congruence of $\alg$.
	Consequently, if $L$ is a finite \tworepr{} lattice that contains a $(0,1)$-sublattice isomorphic to $M_3$,	then $L \cong \Eq{A}$ for some finite set $A$, hence the cardinality of $L$ is a Bell number.
\end{proposition}

\begin{proof}
	Let $\{0_A,\rho,\sigma,\tau,1_A\} \leq \Con{\alg}$ be a sublattice that is isomorphic to $M_3$.
	Then we have $m^-(\rho,\sigma,\tau)=0_A$ and $m^+(\rho,\sigma,\tau)=1_A$ (here $m^-$ and $m^+$ are the median terms introduced in Example~\ref{example_median}), hence Corollary~\ref{cor_intconc} implies that $[0_A,1_A]_{\Eq{A}} \subseteq \Con{\alg}$, thus $\Con{\alg}=\Eq{A}$ indeed.
	The second claim follows immediately: any \twogood{} $\varphi\colon L \hookrightarrow \Eq{A}$ must be surjective.
\end{proof}  

\begin{example}\label{example_M_3^2}

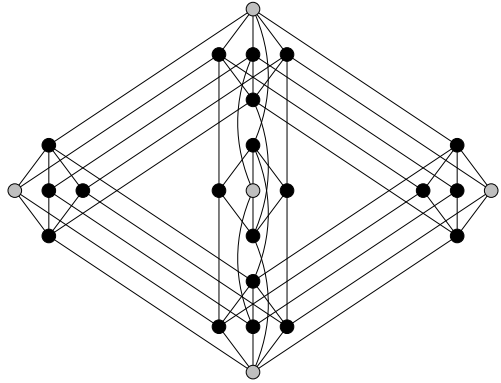
\begin{figure}[ht]
\centering
\begin{tikzpicture}[x=0.9mm,y=0.6mm, line width=0.35pt,
 lab/.style={draw=none, below=3mm, font=\small}]

\newcommand{\vertex}[2]{\filldraw[black] (#1,#2) circle (0.9mm);}
\newcommand{\vertexC}[3]{\filldraw[fill=#3] (#1,#2) circle (0.9mm);}
\newcommand{\lat}[4]{%
 \begin{scope}[shift={(#1,#2)}]
  #4
 \end{scope}}


\lat{0}{0}{M_3\times M_3}{
 \draw (-5,10)--(0,0)--(5,10)--(0,20)--(-5,10);
 \draw (0,0)--(0,20);
 \vertex{-5}{10}\vertex{0}{10}\vertex{0}{20}\vertex{5}{10}}
\lat{-30}{30}{}{
 \draw (-5,10)--(0,0)--(5,10)--(0,20)--(-5,10);
 \draw (0,0)--(0,20);
 \vertex{0}{0}\vertex{0}{10}\vertex{0}{20}\vertex{5}{10}}
\lat{30}{30}{}{
 \draw (-5,10)--(0,0)--(5,10)--(0,20)--(-5,10);
 \draw (0,0)--(0,20);
 \vertex{0}{0}\vertex{-5}{10}\vertex{0}{10}\vertex{0}{20}}
\lat{0}{60}{}{
 \draw (-5,10)--(0,0)--(5,10)--(0,20)--(-5,10);
 \draw (0,0)--(0,20);
 \vertex{0}{0}\vertex{-5}{10}\vertex{0}{10}\vertex{5}{10}}
\lat{0}{30}{}{
 \draw (-5,10)--(0,0)--(5,10)--(0,20)--(-5,10);
 \draw (0,0)--(0,20);
 \vertex{0}{0}\vertex{-5}{10}\vertex{0}{20}\vertex{5}{10}}
\lat{0}{0}{}{%
 \draw (0,0)--(-30,30)--(0,60)--(30,30)--(0,0);}
\lat{-5}{10}{}{%
 \draw (0,0)--(-30,30)--(0,60)--(30,30)--(0,0);}
\lat{0}{10}{}{%
 \draw (0,0)--(-30,30)--(0,60)--(30,30)--(0,0);}
\lat{5}{10}{}{%
 \draw (0,0)--(-30,30)--(0,60)--(30,30)--(0,0);}
\lat{0}{20}{}{%
 \draw (0,0)--(-30,30)--(0,60)--(30,30)--(0,0);}
\lat{0}{0}{}{%
 \draw (0,0) .. controls (3,10) and (3,20) .. (0,30) .. controls (3,40) and (3,50) .. (0,60);}
\lat{0}{10}{}{%
 \draw (0,0) .. controls (-3,10) and (-3,20) .. (0,30) .. controls (-3,40) and (-3,50) .. (0,60);}
\lat{0}{20}{}{%
 \draw (0,0) .. controls (3,10) and (3,20) .. (0,30);
 \draw (0,30) .. controls (3,40) and (3,50) .. (0,60);}
 \draw (-5,10)--(-5,40)--(-5,70);
 \draw (5,10)--(5,40)--(5,70);

 \vertexC{0}{0}{lightgray}
 \vertexC{-35}{40}{lightgray}
 \vertexC{0}{40}{lightgray}
 \vertexC{35}{40}{lightgray}
 \vertexC{0}{80}{lightgray}
\end{tikzpicture}
\caption{The lattice $M_3\times M_3$}\label{fig_M3xM3}
\end{figure}

Let $M_3=\{0,u,v,w,1\}$ and consider the lattice $M_3\times M_3$ (see Figure~\ref{fig_M3xM3}).
The sublattice $\{(0,0),\,(u,u),\,(v,v),\,(w,w),\,(1,1)\} \leq M_3\times M_3$ is isomorphic to $M_3$ (see the gray vertices in Figure~\ref{fig_M3xM3}). 
Hence, if $M_3\times M_3$ were \tworepr, then, by Proposition~\ref{prop_median}, we would have $M_3\times M_3 \cong \Eq{A}$ for some finite set $A$.
However, no partition lattice has exactly $25$ elements; therefore, we can conclude that $M_3\times M_3$ is not \tworepr{}.

On the other hand, it follows from Theorem~\ref{thm_antenna-product} that adjoining an extra top element to $M_3\times M_3$ yields a lattice that is \tworepr{}.
\end{example}

\section{Constructions (not) preserving representability}\label{sect_constructions}

\subsection{Filters, ideals and intervals}

If $\vartheta$ is a congruence of $\alg$, then $\Con{(\alg/\vartheta)} \cong [\vartheta;1_A]_{\Con{\alg}}$.
This immediately implies that principal filters of \tworepr{} lattices are also \tworepr{}.

\begin{proposition}\label{prop_filters}
	If $L$ is a \tworepr{} lattice, then the filter $[u;1_L]$ is also \tworepr{} for every $u \in L$.
\end{proposition}

\begin{remark}\label{rem_M_3^2+1}
	The analogous statement fails for principal ideals.
	Indeed, by Theorem~\ref{thm_antenna-product} and Example~\ref{example_M_3^2}, $M_3^2 \oplus \mathbf{1}$ is \tworepr{}, but $M_3^2$ is not. 
\end{remark}

We will prove that adding an extra top element to a principal ideal of a \tworepr{} lattice yields again a \tworepr{} lattice.

\begin{definition}\label{def_Ftau}
  For $\tau\in\Eq{A}$ define $F_{\tau}=\Set{ f\in T_2(A) }{ \ker f\geq \tau }$.
\end{definition}

\begin{lemma}\label{lemma_Ftau}
	For arbitrary  $\rho,\tau\in\Eq{A}$, we have
	\[
		\rho \lhd F_{\tau} \iff \rho\leq\tau \text{ or } \rho=1_A.
	\]
\end{lemma}

\begin{proof}
	Assume first that $\rho\lhd F_{\tau}$ but $\rho\nleq\tau$; we will prove that $\rho=1_A$. 
	From $\rho\nleq\tau$ it follows that there exists a pair $(a,b)\in\rho\setminus \tau$.
	Let $c\neq d$ be arbitrary elements of $A$ and define the transformation
	\[
	  f\colon A\to A,\quad 
	  f(x)=\begin{cases}
	      c &\text{ if } x \in a/\tau,\\
	      d &\text{ otherwise}.
	     \end{cases}
	\]
	Clearly, $f\in F_{\tau}$, thus $\rho \lhd f$, and this implies that $(c,d)=(f(a),f(b))\in\rho$. 
	Since $c$ and $d$ were arbitrary distinct elements of $A$, it follows that $\rho=A\times A$.
	
	To prove the other direction, assume that $\rho \leq \tau$, and let $(a,b) \in \rho$; then $(a,b) \in \tau$ also holds. 
	For every $f\in F_{\tau}$, we have $\ker f\geq \tau$, hence $f(a)=f(b)$, which obviously guarantees $(f(a),f(b))\in\rho$.
\end{proof}

\begin{theorem}\label{thm_ideals}
	If $L$ is a \tworepr{} lattice and $v \in L \setminus \{ 1_L \}$, then the lattice $[0_L;v]_L \oplus \mathbf{1}$ is also \tworepr{}.
\end{theorem}

\begin{proof}
	If $L$ is \tworepr{}, then there exists an algebra $\alg=(A;F)$ with $F\subseteq T_2(A)$ such that $\Con{\alg} \cong L$. 
	Let $\tau \in \Con{\alg}$ correspond to $v$ under this isomorphism.
	Then the algebra $\alg':=(A;F\cup F_{\tau})$ is also a $2$-valued unary algebra, and Lemma~\ref{lemma_Ftau} shows that $\Con{\alg'}=[0_A;\tau]_{\Con{\alg}} \cup \{1_A\} \cong [0_L;v]_L \oplus \mathbf{1}$.
\end{proof}

\begin{remark}\label{rem_v=1}
	The theorem above is true for $v = 1_L$ as well.
	In this case $[0_L;v]_L \oplus \mathbf{1} = L \oplus \mathbf{1}$, which is \tworepr{} by Corollary~\ref{cor_oplus} if $L$ is \tworepr{}.
\end{remark}

\begin{theorem}\label{thm_intervals}
	If $L$ is a \tworepr{} lattice and $u,v \in L$, then the lattice $[u;v]_L \oplus \mathbf{1}$ is also \tworepr{}.
\end{theorem}

\begin{proof}
	By Proposition~\ref{prop_filters}, the filter $[u;1_L]_L$ is \tworepr{}, and then applying Theorem~\ref{thm_ideals} to this latter lattice we get that $[u;v]_L \oplus \mathbf{1}$ is \tworepr{} (if $v=1_L$, then we use Corollary~\ref{cor_oplus}, as explained in Remark~\ref{rem_v=1}).
\end{proof}

\subsection{Ordinal sum}

\begin{lemma}\label{lem_oplusc} 
Let $L_1$ and $L_2$ be bounded lattices, and set $L=L_1\oplus_w L_2$. Then for every $n \in \N$, $t \in \FL{n}$ and for all $\uv\in L^n$, the following holds:
 \[
 \big[t^-(\uv);t^+(\uv)\big]\subseteq L_1 \text{\; or \;} \big[t^-(\uv);t^+(\uv)\big]\subseteq L_2.
 \]
 Moreover, in the first case there is a tuple $\uv' \in L_1^n$ such that $t^{\mp}(\uv')=t^{\mp}(\uv)$, and similarly, in the second case there is a tuple $\uv' \in L_2^n$ such that $t^{\mp}(\uv')=t^{\mp}(\uv)$.
 \end{lemma}

\begin{proof}
Let $w$ denote the ``gluing point'' as in Definition~\ref{def_ordsum}; furthermore, let $\mathcal{M}$ and $\mathcal{N}$ be the upsets corresponding to $t^-$ and $t^+$, i.e., let $t^-=t^-_{\mathcal{M}}$ and $t^+=t^+_{\mathcal{N}}$, respectively (see Section~\ref{subsect_DNF}). 
It is clear that if $t^-_{\mathcal{M}}(\uv)\in L_2$, then $t^+_{\mathcal{N}}(\uv)\in L_2$, and the claim follows. 

Now suppose that $t^-_{\mathcal{M}}(\uv)<w$; we will show that in this case $t^+_{\mathcal{N}}(\uv)\leq w$. Since
\[
t^-_{\mathcal{M}}(\uv)=\bigvee_{M\in\mathcal{M}}\bigwedge_{i\in M}u_i < w,
\]
we have $\bigwedge_{i\in M}u_i<w$ for every $M\in\mathcal{M}$. This implies that for every $M\in\mathcal{M}$, there exists $i(M)\in M$ such that $u_{i(M)}<w$. 

Recall from Proposition~\ref{prop_MN} that $\mathcal{N}=\Set{\compl{I}}{I\notin\mathcal{M}}$. 
Let $N_0=\Set{i(M)}{M\in\mathcal{M}}$; we claim that $N_0\in\mathcal{N}$. 
Indeed, if $N_0\notin\mathcal{N}$, then $\compl{N_0}\in\mathcal{M}$, hence there exists $i(\compl{N_0})\in \compl{N_0}$ such that $u_{i(\compl{N_0})}<w$. 
However, by the definition of $N_0$, we also have $i(\compl{N_0})\in N_0$, which is a contradiction. 

Therefore $N_0\in\mathcal{N}$, and this allows us to estimate $t^+_{\mathcal{N}}(\uv)$ by the meetand corresponding to $N_0$:
\[
t^+_{\mathcal{N}}(\uv) = \bigwedge_{N\in\mathcal{N}}\bigvee_{i\in N}u_i \leq \bigvee_{i\in N_0}u_i.
\]
For every $i\in N_0$ we have $u_i<w$ by the definition of $N_0$, thus $t^+_{\mathcal{N}}(\uv) \leq w$, and this implies that $[t^-(\uv);t^+(\uv)]\subseteq L_1$.

The last claim of the lemma follows by applying the retraction homomorphisms $L \to L_1,\ u \mapsto u \wedge w$, respectively $L \to L_2,\ u \mapsto u \vee w$.
\end{proof}

\begin{theorem}\label{thm_oplusc}
	If $L_1$ and $L_2$ are \tworepr{} finite lattices, then $L_1\oplus_w L_2$ is also \tworepr{}.
\end{theorem}

\begin{proof}
	Since $L_1$ and $L_2$ are \tworepr{}, there exist finite sets $A_1$ and $A_2$ and \twogood{}s $\varphi_1\colon L_1\hookrightarrow \Eq{A_1}$, $\varphi_2\colon L_2\hookrightarrow \Eq{A_2}$. 
	We can suppose without loss of generality that $A_1$ and $A_2$ have only one common element: $A_1 \cap A_2 = \{a\}$.
	Let $A=A_1\cup A_2$ and consider the following embeddings:
	\begin{align*}
		\iota_1&\colon \Eq{A_1}\hookrightarrow \Eq{A},\quad \rho_1\mapsto \rho_1\cup 0_{A_2\setminus \{a\}},\\
		\iota_2&\colon \Eq{A_2}\hookrightarrow \Eq{A},\quad \rho_2\mapsto \mathop{tr}(A_1^2 \cup \rho_2),
	\end{align*}
	where $\mathop{tr}(\rho)$ denotes the transitive closure of the relation $\rho$. 
	Hence, the blocks of the partition induced by $\iota_1(\rho_1)$ are precisely the blocks of $\rho_1$ together with the singleton blocks $\{x\}$, where $x\in A_2\setminus \{a\}$. 
	Moreover, the blocks of the partition induced by $\iota_2(\rho_2)$ are $A_1\cup (a/\rho_2)$ and $x/\rho_2$, where $x\in A_2\setminus(a/\rho_2)$.
	Let $\vartheta \in Eq(A)$ be the equivalence relation whose only potentially non-singleton block is $A_1$.
	Then the image of $\iota_1$ is $[0_A;\vartheta]_{\Eq{A}}$, and the image of $\iota_2$ is $[\vartheta;1_A]_{\Eq{A}}$.
	
	We obtain the following $(0,1)$-embedding:
	\[
		\varphi\colon L_1\oplus_w L_2\hookrightarrow \Eq{A},\quad 
		\varphi(x)=\begin{cases}
			\iota_1(\varphi_1(x)) &\text{ if } x \in L_1,\\
			\iota_2(\varphi_2(x)) &\text{ if } x\in L_2.
		\end{cases}
	\]
	If $x=w$, the common element of $L_1$ and $L_2$, then $\iota_1(\varphi_1(x)) = \iota_2(\varphi_2(x)) = \vartheta$.
	
	Now we show that the embedding $\varphi$ is a \twogood{}. By Corollary~\ref{cor_intabst}, it suffices to prove that for every $t \in \FL{n}$ and all $u_1,\dots,u_n\in L_1\oplus_w L_2$, we have 
	\[ 
		\big[\varphi(t^-(\uv));\varphi(t^+(\uv))\big]_{\Eq{A}}\subseteq \varphi(L_1\oplus_w L_2).
	\]
	By Lemma~\ref{lem_oplusc}, either $t^-(\uv),t^+(\uv)\in L_1$ or $t^-(\uv),t^+(\uv)\in L_2$, and we may assume $\uv \in L_1^n$ or $\uv \in L_2^n$, respectively.
	The two cases are similar, so we only give the details for the first case. 
	Let $\sigma\in\Eq{A}$ satisfy $\varphi(t^-(\uv))\leq \sigma\leq \varphi(t^+(\uv))$. 
	For $x\in L_1$, we have $\varphi(x)=\iota_1(\varphi_1(x))$, thus $\iota_1(\varphi_1(t^-(\uv)))\leq \sigma\leq\iota_1(\varphi_1 (t^+(\uv)))$.
	The image of $\iota_1$ is convex (namely, the interval $[0_A;\vartheta]_{\Eq{A}}$), hence it follows that $\sigma=\iota_1(\sigma_1)$ for some $\sigma_1\in \big[\varphi_1(t^-(\uv));\varphi_1(t^+(\uv))\big]_{\Eq{A_1}}$. 
	Since $\varphi_1$ is a \twogood{} of $L_1$ into $\Eq{A_1}$, there exists $v\in L_1$ such that $\sigma_1=\varphi_1(v)$.
	This implies $\sigma=\varphi(v) \in \varphi(L_1\oplus_w L_2)$, as claimed. 
\end{proof}

Using Remark~\ref{rem_oplus}, Corollary~\ref{cor_dist} and Theorem~\ref{thm_oplusc}, we obtain the following corollary.

\begin{corollary}\label{cor_oplus}
  If $L_1$ and $L_2$ are \tworepr{} finite lattices, then $L_1\oplus L_2$ is also \tworepr{}.
\end{corollary}

Proposition~\ref{prop_filters} also gives a partial converse: if $L_1\oplus L_2$ (or $L_1\oplus_w L_2$) is \tworepr{}, then so is $L_2$ (however, as Remark~\ref{rem_M_3^2+1} shows, $L_1$ is not necessarily \tworepr{}).

\subsection{Direct product}

As Example~\ref{example_M_3^2} shows, the class of \tworepr{} lattices is not closed under direct products: $M_3 \times M_3$ is not \tworepr{}, though $M_3$ is clearly \tworepr{}, as $M_3\cong \Eq{[3]} = \Con{([3];\emptyset)}$.
We generalize this example to direct products of larger partition lattices in the next proposition.

\begin{proposition}\label{prop_Eq(A1)xEq(A2)}
	If $m_1,m_2 \geq 3$, then $\Eq{[m_1]} \times \Eq{[m_2]}$ is not \tworepr{}.
\end{proposition}

\begin{proof}
	First we show that if $m \geq 3$, then $\Eq{[m]}$ contains a $(0,1)$-sublattice $\{0_{[m]},\rho,\sigma,\tau,1_{[m]}\}$ that is isomorphic to $M_3$.
	We distinguish two cases according to the parity of $m$.
	If $m$ is even, say $m=2k$, then let $\rho,\sigma,\tau \in \Eq{[m]}$ be the equivalence relations whose blocks are as follows (the left-hand diagram in Figure~\ref{fig_uvw-even-odd} illustrates the case $m=10$):
	\[\renewcommand{\arraystretch}{1.3}
	\begin{array}{l@{\,}c@{}l}
		\text{blocks of } &\rho&\colon\; \{1,k+1\},\, \{2,k+2\}, \dots, \{k-1,2k-1\},\, \{k,2k\};\\
		\text{blocks of } &\sigma&\colon\; \{1,k+2\},\, \{2,k+3\}, \dots, \{k-1,2k\},\, \{k\},\, \{k+1\};\\
		\text{blocks of } &\tau&\colon\; \{1,\dots,k\},\, \{k+1,\dots,2k\}.
	\end{array}
	\]
	If $m$ is odd, say $m=2k+1$, then let $\rho,\sigma,\tau \in \Eq{[m]}$ be the equivalence relations whose blocks are as follows (the right-hand diagram in Figure~\ref{fig_uvw-even-odd} illustrates the case $m=11$):
	\[\renewcommand{\arraystretch}{1.3}
	\begin{array}{l@{\,}c@{}l}
		\text{blocks of } &\rho&\colon\; \{1,k+1\},\, \{2,k+2\}, \dots, \{k-1,2k-1\},\, \{k,2k\},\, \{2k+1\};\\
		\text{blocks of } &\sigma&\colon\; \{1,k+2\},\, \{2,k+3\}, \dots, \{k-1,2k\},\, \{k, 2k+1\},\, \{k+1\};\\
		\text{blocks of } &\tau&\colon\; \{1,\dots,k\},\, \{k+1,\dots,2k+1\}.
	\end{array}
	\]
	It is straightforward to verify that $\rho \wedge \sigma = \rho \wedge \tau = \sigma \wedge \tau = 0_{[m]}$ and $\rho \vee \sigma = \rho \vee \tau = \sigma \vee \tau = 1_{[m]}$, thus the sublattice $\{0_{[m]},\rho,\sigma,\tau,1_{[m]}\}$ is indeed isomorphic to $M_3$.
	
	Now let us consider the direct product $L := \Eq{[m_1]} \times \Eq{[m_2]}$.
	If $\rho_1,\sigma_1,\tau_1 \in \Eq{[m_1]}$ and $\rho_2,\sigma_2,\tau_2 \in \Eq{[m_2]}$ are constructed as above, then the elements $\rho=(\rho_1,\rho_2)$, $\sigma=(\sigma_1,\sigma_2)$, $\tau=(\tau_1,\tau_2)$ yield a sublattice $\{0_L,\rho,\sigma,\tau,1_L\} \leq L$ that is isomorphic to $M_3$.
	Therefore, by Proposition~\ref{prop_median}, if $L$ is \tworepr{}, then $L$ is isomorphic to a partition lattice.
	However, this is not possible, since partition lattices are directly indecomposable.
\end{proof}

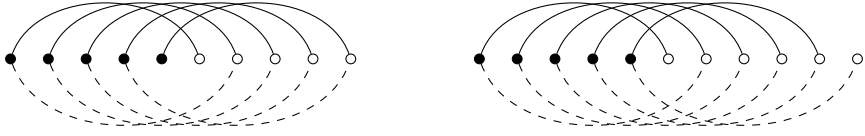
\begin{figure}[t]
	\centering
	\begin{tikzpicture}[x=0.5cm,y=0.5cm, every node/.style={font=\small}]
	
	\begin{scope}[xshift=0cm]
	
	\foreach \i in {1,...,10}{
	    \pgfmathtruncatemacro{\x}{\i-1}
	    \ifnum\i<6
	        \node[circle, draw, fill=black, inner sep=1.3pt] (a\i) at (\x,0) {};
	    \else
	        \node[circle, draw, fill=white, inner sep=1.3pt] (a\i) at (\x,0) {};
	    \fi
	}
	
	\foreach \p/\q in {1/6,2/7,3/8,4/9,5/10}{
	    \draw (a\p) to[out=70,in=110] (a\q);
	}
	
	\foreach \p/\q in {1/7,2/8,3/9,4/10}{
	    \draw[dashed] (a\p) to[out=-70,in=-110] (a\q);
	}
	
	\end{scope}
	
	\begin{scope}[xshift=6.2cm]
	
	\foreach \i in {1,...,11}{
	    \pgfmathtruncatemacro{\x}{\i-1}
	    \ifnum\i<6
	        \node[circle, draw, fill=black, inner sep=1.3pt] (b\i) at (\x,0) {};
	    \else
	        \node[circle, draw, fill=white, inner sep=1.3pt] (b\i) at (\x,0) {};
	    \fi
	}
	
	\foreach \p/\q in {1/6,2/7,3/8,4/9,5/10}{
	    \draw (b\p) to[out=70,in=110] (b\q);
	}
	
	\foreach \p/\q in {1/7,2/8,3/9,4/10,5/11}{
	    \draw[dashed] (b\p) to[out=-70,in=-110] (b\q);
	}
	
	\end{scope}
	
	\end{tikzpicture}
	\caption{Illustration of the equivalence relations $\rho$, $\sigma$, and $\tau$ in the proof of Proposition~\ref{prop_Eq(A1)xEq(A2)} in the cases $m=10$ (left) and $m=11$ (right). Filled and unfilled points indicate the two blocks of $\tau$; solid upper arcs indicate the non-singleton blocks of $\rho$; dashed lower arcs indicate the non-singleton blocks of $\sigma$.}
	\label{fig_uvw-even-odd}
\end{figure}

\begin{lemma}\label{lem_rho_lifting}
  Let $\varphi\colon L\hookrightarrow\Eq{[m]}$ be a \twogood{}, where $m$ is a natural number. 
  For every $\rho \in\Eq{[m]}$, define relations $\rho_0,\rho_1$ on $[m+1]$ by
  \begin{align*}
    \rho_0 &= \rho \cup \{m+1\}^2;\\
    \rho_1 &= \mathop{tr}\big(\rho \cup \{m,m+1\}^2\big),
  \end{align*}
  where $\mathop{tr}$ denotes the transitive closure.
  Then 
  \[
    \hat\varphi\colon L\times \mathbf{2} \hookrightarrow\Eq{[m+1]},\quad 
    \hat\varphi(u,v)=\varphi(u)_v
  \]
 is also a \twogood{}.
\end{lemma}

\begin{proof} 
	It is straightforward to check that the relations $\rho_0$ and $\rho_1$ are equivalence relations for every $\rho\in\Eq{[m]}$, and the mappings 
	\begin{align*}
		\iota_0\colon \Eq{[m]}\hookrightarrow\Eq{[m+1]},\ &\rho\mapsto\rho_0,\\
		\iota_1\colon \Eq{[m]}\hookrightarrow\Eq{[m+1]},\ &\rho\mapsto\rho_1
	\end{align*}
	are embeddings.
	(Observe that $\rho_0$ adds the singleton block $\{m+1\}$ to $\rho$, while $\rho_1$ is obtained by appending $m+1$ to the block $m/\rho$. Note also the similarity to the embeddings $\iota_1$ and $\iota_2$ defined in the proof of Theorem~\ref{thm_oplusc}.) 
	Moreover, for arbitrary equivalence relations $\rho,\sigma\in\Eq{[m]}$, we have 
	\begin{subequations}\label{F:rho} 
	\begin{align}
		\rho_0\wedge\sigma_1 &= (\rho\wedge\sigma)_0 = \rho_0\wedge\sigma_0,  \label{F:rho-a}\\
		\rho_0\vee\sigma_1 &= (\rho\vee\sigma)_1 = \rho_1\vee\sigma_1.    \label{F:rho-b}
	\end{align}
	\end{subequations}
	
	To prove that $\hat{\phi}$ is a homomorphism, choose arbitrary elements $(u_1,v_1)$ and $(u_2,v_2)$ in $L\times\mathbf{2}$. Then 
	\begin{align*}
		\hat{\phi}((u_1,v_1)\wedge(u_2,v_2)) 
		&= \hat{\phi}(u_1\wedge u_2,v_1\wedge v_2)\\
		&= \phi(u_1\wedge u_2)_{v_1\wedge v_2}\\
		&= (\phi(u_1)\wedge\phi(u_2))_{v_1\wedge v_2}\\
		&= \begin{cases}
			\phi(u_1)_0\wedge\phi(u_2)_0,&\mbox{if } 0\in\{v_1,v_2\},\\
			\phi(u_1)_1\wedge\phi(u_2)_1,&\mbox{if } v_1=v_2=1,
		\end{cases}\\
		\noalign{\noindent and }
		\hat{\phi}(u_1,v_1)\wedge\hat{\phi}(u_2,v_2) 
		&= \phi(u_1)_{v_1}\wedge\phi(u_2)_{v_2}.
	\end{align*}
	The equality of $\hat{\phi}((u_1,v_1)\wedge(u_2,v_2))$ and $\hat{\phi}(u_1,v_1)\wedge\hat{\phi}(u_2,v_2)$ is clear if $v_1=v_2$, otherwise it follows from \eqref{F:rho-a}. 
	For the join operation, we have 
	\begin{align*}
		\hat{\phi}((u_1,v_1)\vee(u_2,v_2)) 
		&= \hat{\phi}(u_1\vee u_2,v_1\vee v_2)\\
		&= \phi(u_1\vee u_2)_{v_1\vee v_2}\\
		&= (\phi(u_1)\vee\phi(u_2))_{v_1\vee v_2}\\
		&= \begin{cases}
			\phi(u_1)_1\vee\phi(u_2)_1,&\mbox{if } 1\in\{v_1,v_2\},\\
			\phi(u_1)_0\vee\phi(u_2)_0,&\mbox{if } v_1=v_2=0,
		\end{cases}\\
		\noalign{\noindent and }
		\hat{\phi}(u_1,v_1)\vee\hat{\phi}(u_2,v_2) 
		&= \phi(u_1)_{v_1}\vee\phi(u_2)_{v_2}.
	\end{align*}
	The equality of $\hat{\phi}((u_1,v_1)\vee(u_2,v_2))$ and $\hat{\phi}(u_1,v_1)\vee\hat{\phi}(u_2,v_2)$ 
	is obvious if $v_1=v_2$, otherwise it follows from \eqref{F:rho-b}. 
	Therefore, the mapping $\hat{\phi}$ is a homomorphism, and thus it is a $(0,1)$-embedding (injectivity is clear, as well as the images of the boundary elements). 
	
	By Corollary~\ref{cor_intabst}, we need to prove that for every natural number $n$, every term $t\in\FL{n}$, and all elements $(u_1,v_1),\dots,(u_n,v_n)\in L\times\mathbf{2}$, the interval
	\[
		I:=\big[\hat{\varphi}\big(t^-((u_1,v_1),\dots,(u_n,v_n))\big);\hat{\varphi}\big(t^+((u_1,v_1),\dots,(u_n,v_n))\big)\big]_{\Eq{[m+1]}}  
	\]
	is a subset of $\hat{\varphi}(L\times\mathbf{2})$. 
	Let us compute the endpoints of $I$:
	\begin{equation}\label{Eq:Formula}
		\hat{\phi}\big(t^{\mp}((u_1,v_1),\dots,(u_n,v_n))\big)
		= \hat{\phi}\big(t^{\mp}(\uv),t^{\mp}(\vv)\big)
		= \big(\phi(t^{\mp}(\uv))\big)_{\ell},
	\end{equation}
	where $\ell:=t^-(\vv)=t^+(\vv)$, as $\mathbf{2}$ is a distributive lattice.
	Setting $\kappa:=\phi(t^-(\uv))$ and $\mu:=\phi(t^+(\uv))$, we can write the interval as
	$I = [\kappa_{\ell};\mu_{\ell}]_{\Eq{[m+1]}}$.
	Let $\xi\in I$, i.e., $\kappa_{\ell}\subseteq\xi\subseteq\mu_{\ell}$, and let us verify that $\xi \in\hat{\phi}(L\times\mathbf{2})$.
	The image of $\iota_\ell$ is a convex set (in fact, an interval), thus there is an equivalence relation $\lambda\in[\kappa;\mu]_{\Eq{[m]}}$ such that $\xi=\lambda_{\ell}$.
	Since $\phi$ is a \twogood{}, we have $[\kappa;\mu]_{\Eq{[m]}}\subseteq\phi(L)$, hence $\lambda = \phi(\eta)$ for some $\eta \in L$.
	We can conclude that $\xi=\hat{\phi}(\eta,\ell)\in\hat{\phi}(L\times\mathbf{2})$, and this completes the proof.
\end{proof}

\begin{theorem}\label{thm_dirdist}
	If $L$ is a finite \tworepr{} lattice and $D$ is a finite distributive lattice, then $L\times D$ is also \tworepr{}.
\end{theorem}

\begin{proof}
	According to Birkhoff's representation theorem (see \cite[Chapter~II, \S1.2, Corollary~109]{Gratzer11}), we may regard $D$ as a $(0,1)$-sublattice of $\mathbf 2^r$ for some natural number $r$.
	Since $L$ is \tworepr{}, Lemma~\ref{lem_rho_lifting} implies that $L\times\mathbf 2^r$ is \tworepr{}, i.e., there exists a \twogood{} $\varphi\colon L \times \mathbf{2}^r \hookrightarrow \Eq{A}$. 
	We show that the restriction $\varphi\vert_{L\times D}\colon L\times D \hookrightarrow\Eq{A}$ is also a \twogood{}.
	Let $t \in \FL{n}$ and let $(u_i,v_i)\in L\times D$ for $i\in [n]$. 
	Then $t^{\mp}((u_1,v_1),\dots,(u_n,v_n))=(t^{\mp}(\uv),t^{\mp}(\vv))$.
	We have to verify that whenever $\sigma\in\Eq{A}$ satisfies 
	\[
		\varphi(t^-(\uv),t^-(\vv))\leq\sigma\leq \varphi (t^+(\uv),t^+(\vv)),
	\]
	there exists $(u,v)\in L\times D$ such that $\sigma=\varphi(u,v)$.
	We know that $\sigma=\varphi(u,v)$ for some $(u,v)\in L\times \mathbf 2^r$, as $\varphi$ is a \twogood{}, so we only have to show that $v \in D$.
	By the monotonicity of the map $\varphi^{-1}$, we have
	\[
		(t^-(\uv),t^-(\vv))\leq(u,v)\leq (t^+(\uv),t^+(\vv)).
	\]
	Distributivity of $D$ implies that $t^-(\vv)=t^+(\vv)$, hence $v=t^-(\vv)=t^+(\vv)\in D$, and thus $(u,v)\in L\times D$, as needed.
\end{proof}

\begin{proposition}\label{prop_dir}
	If the direct product $L_1\times L_2$ is \tworepr{}, then both $L_1$ and $L_2$ are \tworepr{}.
\end{proposition}

\begin{proof}
	Observe that $L_1\times \{1_{L_2}\}$ and $\{1_{L_1}\}\times L_2$ are principal filters of $L_1\times L_2$ that are isomorphic to $L_1$ and $L_2$, respectively. Therefore, both of them are \tworepr{} by Proposition~\ref{prop_filters}.
\end{proof}

Combining Theorem~\ref{thm_dirdist} and Proposition~\ref{prop_dir}, we obtain the following.

\begin{corollary}
	If $D$ is a finite distributive lattice, then $L\times D$ is \tworepr{} if and only if $L$ is \tworepr{}.
\end{corollary}

\begin{example}
	Here are some examples of non-distributive lattices whose direct products are \tworepr{}.
	The corresponding $2$-valued unary algebras were found by computer search.
	\begin{itemize}
		\item The lattice $M_3\times N_5$ is \tworepr{} with $\abs{A}=6$.
		\item The lattice $N_5\times N_5$ is \tworepr{} with $\abs{A}=7$.
	\end{itemize}
	\end{example}

\begin{theorem}\label{thm_antenna-product}
If $L_1$ and $L_2$ are \tworepr{} lattices, then so is the lattice $(L_1\times L_2)\oplus \mathbf{1}$.
\end{theorem}

\begin{proof}
Since $L_1$ and $L_2$ are \tworepr{}, there exist finite algebras $\alg_i=(A_i,F_i)$ with $F_i\subseteq T_2(A_i)$ such that $\Con(\alg_i)\cong L_i$ for $i=1,2$.
Let $A=A_1\mathop{\dot\cup} A_2$ and let $\tau\in\Eq{A}$ be the equivalence relation whose blocks are $A_1$ and $A_2$. For every $f_1\in F_1$ with $f_1(A_1)=\{a_1,b_1\}$, we define $\hat{f}_1\colon A\to A$ by 
\[
\hat{f}_1(x)=\begin{cases}
  f_1(x) &\text{ if } x \in A_1,\\
  a_1 &\text{ if } x\in A_2.
\end{cases}
\]
Similarly, for every $f_2\in F_2$, we define $\hat{f}_2$ by interchanging the roles of $1$ and $2$.

Consider the algebra $\alg=(A;F_\tau\cup \hat{F}_1 \cup \hat{F}_2)$, 
where $F_\tau$ is defined in Definition~\ref{def_Ftau} and $\hat{F}_i = \big\{\hat{f}_i : f_i \in F_i\big\}$ for $i=1,2$.
We claim that
\begin{equation}\label{equ_antenna}
\rho \in \Con(\alg) \iff 
\rho = 1_A 
\text{ or }
\bigl(\rho \leq \tau \text{ and } \rho\vert_{A_i}\in\Con(\alg_i)
 \text{ for } i=1,2\bigr).
\end{equation}
Indeed, if $\rho\in\Con(\alg)$, then $\rho \lhd F_\tau$, hence by Lemma~\ref{lemma_Ftau} we have $\rho=1_A$ or $\rho\leq\tau$.
In the latter case, $\rho \lhd \hat{F}_i$ if and only if $\rho\vert_{A_i}\lhd F_i$, and this proves \eqref{equ_antenna}.

Clearly, it follows from \eqref{equ_antenna} that $\Con{\alg} \cong (\Con{\alg_1} \times \Con{\alg_2}) \oplus \mathbf{1} \cong (L_1 \times L_2) \oplus \mathbf{1}$.
All basic operations of $\alg$ are $2$-valued, thus the proof is complete.  
\end{proof}

\section{Concluding remarks and open problems}\label{sect_conclusion}

We provided a characterization of congruence lattices of $2$-valued unary algebras and we showed that the class of lattices that are isomorphic to such congruence lattices is closed under certain constructions. 
We gave several examples and non-examples, and in the appendix we present results of two systematic computer searches: we list all congruence lattices of $2$-valued unary algebras of sizes up to four (see Appendix~\ref{sect_four-element-sets}) and we classify lattices with at most seven elements with respect to \tworepry{} (see Appendix~\ref{sect_seven-element-lattices}).

Of course, many questions remain open and deserve further study.
The first of these suggests itself immediately.

\begin{problem}
	Generalize the results of this paper to $k$-valued unary algebras for $k=3,4,\dots$, or even to arbitrary finite algebras.
\end{problem}

The class of \tworepr{} lattices is closed under taking filters, but it is not closed under taking ideals. The counterexample $M_3^2 \oplus \bf{1}$ given in Remark~\ref{rem_M_3^2+1} is quite large; perhaps smaller ones do exist.

\begin{problem}
	Find the smallest example of a \tworepr{} lattice having an ideal that is not \tworepr{}.
\end{problem}

Theorem~\ref{thm_dirdist} raises the following problem.

\begin{problem}
	Is there a nondistributive finite lattice $N$ such that for every \tworepr{} finite lattice $L$, the direct product $L \times N$ is also \tworepr{}? 
\end{problem}

For a given finite lattice $L$, we may ask not only about the existence of a finite $2$-valued unary algebra with a congruence lattice isomorphic to $L$, but also about the possible cardinalities of such algebras: let 
\[
	\spec_2(L) = \big\{ \abs{A}: \Con{(A;F)} \cong L \text{ for some } F \subseteq T_2(A) \big\}.
\]

\begin{problem}
	Determine $\spec_2(L)$ for every \tworepr{} lattice $L$.  
\end{problem}

We present two partial results towards the solution of the above problem.
 
\begin{proposition}\label{prop_spec2(Eq)}
	If $m \geq 3$, then $\spec_2\big(\Eq{[m]}\big) = \{m\}$.
\end{proposition}
\begin{proof}
	We can apply Proposition~\ref{prop_median} to the lattice $\Eq{[m]}$, as we have seen in the proof of Proposition~\ref{prop_Eq(A1)xEq(A2)} that it contains a $(0,1)$-sublattice isomorphic to $M_3$.
	Thus if $\Con{\alg} \cong \Eq{[m]}$ for some algebra $\alg=(A;F)$ with $F \subseteq T_2(A)$, then $\Con{\alg} = \Eq{A}$, hence $\abs{A}=m$. 
\end{proof}

\begin{proposition}\label{prop_spec2(D)}
	For every nontrivial finite distributive lattice $D$, we have
	\[
		\spec_2(D) = \big\{h(D)+1,h(D)+2,\dots\big\},
	\]
	where $h(D)$ denotes the height of the lattice $D$.
\end{proposition}
\begin{proof}
	First let us recall that for distributive lattices, \twogood{}s coincide with $(0,1)$-embeddings, as condition \ref{item_interval-conditionc} of Corollary~\ref{cor_intconc} holds automatically.
	Recall also that by Corollary~\ref{cor_dist}, $\spec_2(D)$ is not empty (but that corollary doesn't give any information about the elements of this set).
	It is clear that the least element of $\spec_2(D)$ must be at least $h(D)+1$, since the heigth of $\Eq{[m]}$ is $m-1$.
	
	Let $J(D)$ denote the set of join-irreducible elements of $D$. According to Birkhoff's representation theorem (see \cite[Chapter~II, \S1.2, Theorem~107]{Gratzer11}), $D$ is isomorphic to the lattice of downsets of the poset $J(D)$, thus the height of $D$ is $\abs{J(D)}$ (see \cite[Chapter~II, \S1.2, Corollary~112]{Gratzer11}).
	
	In order to prove that $h(D)+1 \in \spec_2(D)$, let us append an external least element $\bot$ to the poset $J(D)$.
	For an arbitrary element $u \in D$, let $\varphi(u)$ denote the equivalence relation on $J(D) \cup \{\bot\}$ whose only potentially non-singleton block is $S \cup \{\bot\}$, where $S$ denotes the downset of $J(D)$ corresponding to $u$.
	Then $\varphi\colon D \to \Eq{J(D) \cup \{\bot\}}$ is a $(0,1)$-embedding, hence $\abs{J(D)}+1 \in \spec_2(D)$ indeed.
	
	Now it suffices to prove that $m \in \spec_2(D)$ implies $m+1 \in \spec_2(D)$ for all natural numbers $m$.
	Assume that $\varphi\colon D\hookrightarrow \Eq{[m]}$ is a \twogood{}. 
	By the prime ideal theorem (see \cite[Chapter~II, \S1.4, Theorem~115]{Gratzer11}), $D$ can be decomposed into the union of a prime ideal $I$ and a prime filter $F$ that are disjoint from each other.
	Consider the embeddings $\iota_0$ and $\iota_1$ used in the proof of Lemma~\ref{lem_rho_lifting}.
	Using \eqref{F:rho-a} and \eqref{F:rho-b} it is straightforward to verify that the following map $\psi\colon D\hookrightarrow \Eq{[m+1]}$ is a $(0,1)$-embedding:
	\[
		\psi(x) = 
		\begin{cases}
			\varphi(x)_0 & \text{ if } x \in I,\\
			\varphi(x)_1 & \text{ if } x \in F.
		\end{cases}
	\]
	Since $D$ is distributive, every $(0,1)$-embedding of $D$ into a partition lattice is a \twogood{}; hence $m+1\in\operatorname{Sp}_2(D)$.
\end{proof}

Proposition~\ref{prop_spec2(Eq)} implies that there is no \twogood{} $\Eq{[m]} \to \Eq{[m']}$ whenever $3 \leq m < m'$.
However, this does not rule out the existence of a $(0,1)$-embedding $\Eq{[m]} \to \Eq{[m']}$, and this leads us to the following purely lattice-theoretic problem.

\begin{problem}\label{problem_(0,1)}
	Determine all pairs of natural numbers $m < m'$ for which a $(0,1)$-embedding $\Eq{[m]} \to \Eq{[m']}$ exists.  
\end{problem}
Czédli \cite{Cz} has provided results towards the solution of Problem~\ref{problem_(0,1)}.
Among others, he proved that there is no $(0,1)$-embedding of $\Eq{[m]}$ into either $\Eq{[m+1]}$ or $\Eq{[m+2]}$ for $m \geq 4$ and that there is no $(0,1)$-embedding of $\Eq{[3k+1]}$ into $\Eq{[3k+5]}$ for any positive integer $k$.
In particular, there is no $(0,1)$-embedding of $\Eq{[4]}$ into $\Eq{[8]}$, although $(0,1)$-embeddings of $\Eq{[4]}$ into $\Eq{[7]}$ and $\Eq{[9]}$ do exist \cite{Cz}.
 
Since for distributive lattices \twogood{}s coincide with $(0,1)$-embeddings, Proposition~\ref{prop_spec2(D)} tells us that if $D \leq \Eq{[m]}$ is small enough so that it doesn't contain any sublattice isomorphic to $M_3$ or $N_5$, then $D$ embeds into $\Eq{[m+1]}$ in a $(0,1)$-preserving way.
On the other hand, according to the above mentioned results of Czédli \cite{Cz}, $\Eq{[m]}$ is already too large to admit a $(0,1)$-embedding into $\Eq{[m+1]}$.
This calls for a characterization of sublattices of $\Eq{[m]}$ that are ``small enough'' in the above sense.

\begin{problem}
	Determine sublattices $L \leq \Eq{[m]}$ that are maximal with respect to the property of admitting a $(0,1)$-embedding $L \to \Eq{[m+1]}$.  
\end{problem}

Finally, we pose a semigroup-theoretic problem related to the ``other side'' of the Galois connection between equivalence relations and $2$-valued unary operations.

\begin{problem}
	Determine Galois-closed subsets $M \subseteq T_2(A)$, i.e., subsets satisfying $M = M^{*\circledast}$.
\end{problem}

\clearpage

\appendix

\section{Congruence lattices of $2$-valued unary algebras with at most four elements}\label{sect_four-element-sets}
Let $\alg=(A;F)$ be a unary algebra with $F\subseteq T_2(A)$ and $1\leq\abs{A}\leq4$. 
Then $\Con{\alg}$ is isomorphic to one of the lattices in Figure~\ref{Fig:ConLat_on_4-1}.

\bigskip

\bigskip

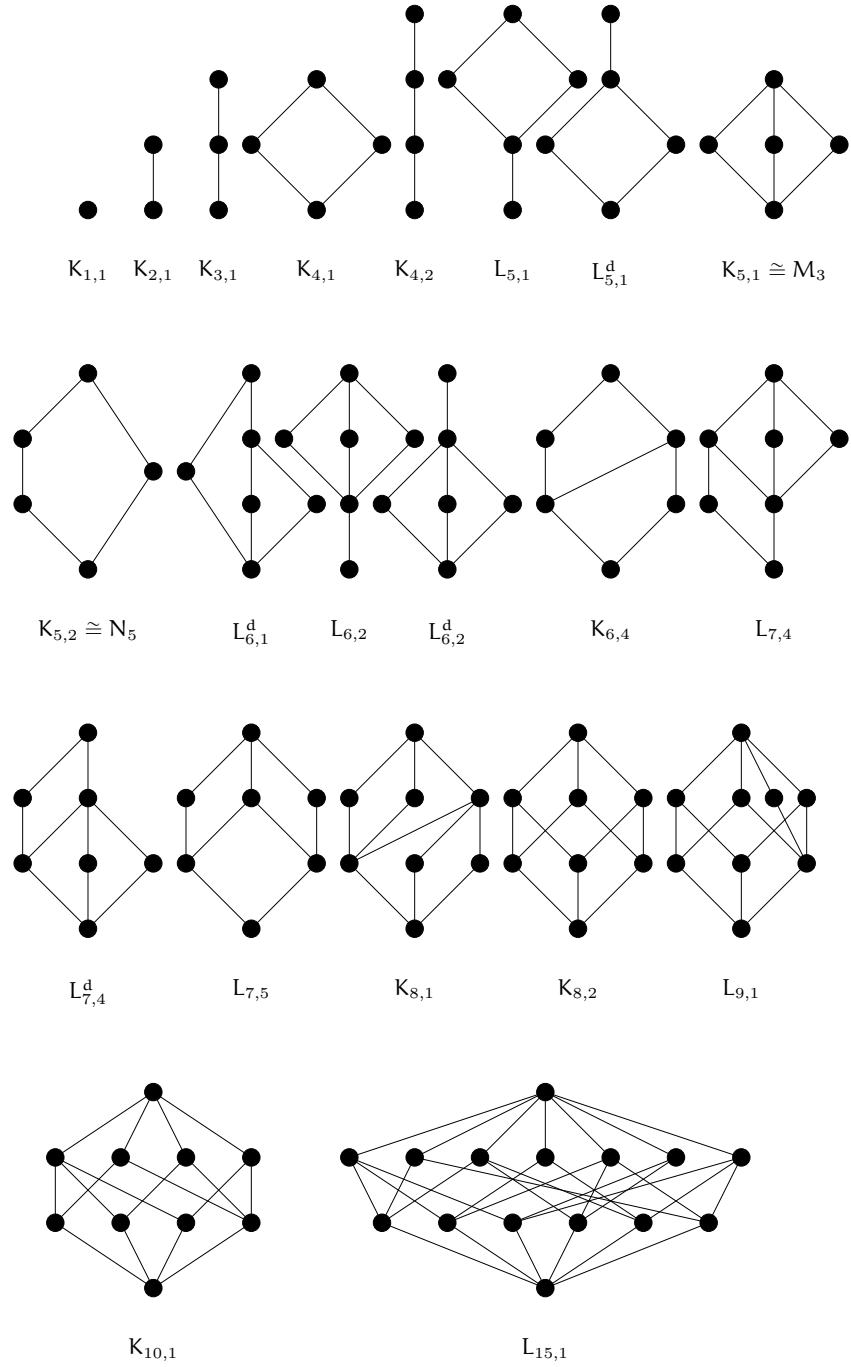
\begin{figure}[ht]
\centering
\scalebox{0.96}{
\begin{tikzpicture}[x=0.9mm,y=0.9mm, line width=0.35pt,
  lab/.style={draw=none, fill=none, below=3mm, font=\small}]

\newcommand{\vertex}[2]{\filldraw[black] (#1,#2) circle (1.2mm);}
\newcommand{\vertexL}[5]{\filldraw[black] (#1,#2) circle (1.2mm); \node at (#4,#5) {$#3$};}
\newcommand{\lat}[4]{%
  \begin{scope}[shift={(#1,#2)}]
    #4
    \node[lab] at (0,-3) {$#3$};
  \end{scope}
}


\lat{0}{0}{K_{1,1}}{%
    \vertex{0}{0};}

\lat{10}{0}{K_{2,1}}{%
    \draw (0,0)--(0,10);
    \vertex{0}{0}\vertex{0}{10}}

\lat{20}{0}{K_{3,1}}{%
    \draw (0,0)--(0,10)--(0,20);
    \vertex{0}{0}\vertex{0}{10}\vertex{0}{20}}

\lat{35}{0}{K_{4,1}}{
    \draw (0,0)--(-10,10)--(0,20)--(10,10)--(0,0);
    \vertex{0}{0}\vertex{-10}{10}\vertex{0}{20}\vertex{10}{10}}

\lat{50}{0}{K_{4,2}}{%
    \draw (0,0)--(0,30);
    \vertex{0}{0}\vertex{0}{10}\vertex{0}{20}\vertex{0}{30}}

\lat{65}{0}{L_{5,1}}{%
    \draw (0,0)--(0,10)--(-10,20)--(0,30)--(10,20)--(0,10);
    \vertex{0}{0}\vertex{0}{10}\vertex{-10}{20}\vertex{0}{30}\vertex{10}{20}}

\lat{80}{0}{L_{5,1}^d}{%
    \draw (0,0)--(-10,10)--(0,20)--(0,30); 
    \draw (0,0)--( 10,10)--(0,20);
    \vertex{0}{0}\vertex{-10}{10}\vertex{10}{10}\vertex{0}{20}\vertex{0}{30}}

\lat{105}{0}{K_{5,1}\cong M_3}{%
    \draw (0,0)--(-10,10)--(0,20); 
    \draw (0,0)--(  0,10)--(0,20);
    \draw (0,0)--( 10,10)--(0,20); 
    \vertex{0}{0}\vertex{-10}{10}\vertex{0}{10}\vertex{10}{10}\vertex{0}{20}}


\lat{0}{-55}{K_{5,2}\cong N_5}{%
    \draw (0,0)--(-10,10)--(-10,20)--(0,30)--(10,15)--(0,0);
    \vertex{0}{0}\vertex{-10}{10}\vertex{-10}{20}\vertex{0}{30}\vertex{10}{15}}

\lat{25}{-55}{L_{6,1}^d}{%
    \draw (0,0)--(-10,15)--(0,30);
    \draw (0,0)--(0,10)--(0,20)--(0,30);
    \draw (0,0)--(10,10)--(0,20);
    \vertex{0}{0}\vertex{0}{10}\vertex{10}{10}\vertex{-10}{15}\vertex{0}{20}\vertex{0}{30}}

\lat{40}{-55}{L_{6,2}}{%
    \draw (0,0)--(0,10)--(-10,20)--(0,30); 
    \draw (0,10)--(0,20)--(0,30);
    \draw (0,10)--(10,20)--(0,30);
    \vertex{0}{0}\vertex{0}{10}\vertex{-10}{20}\vertex{0}{20}\vertex{10}{20}\vertex{0}{30}}

\lat{55}{-55}{L_{6,2}^d}{%
    \draw (0,0)--(-10,10)--(0,20)--(0,30); 
    \draw (0,0)--(0,10)--(0,20);
    \draw (0,0)--(10,10)--(0,20);
    \vertex{0}{0}\vertex{-10}{10}\vertex{0}{10}\vertex{10}{10}\vertex{0}{20}\vertex{0}{30}}

\lat{80}{-55}{K_{6,4}}{%
    \draw (0,0)--(-10,10)--(-10,20)--(0,30)--(10,20)--(10,10)--(0,0); 
    \draw (-10,10)--(10,20);
    \vertex{0}{0}\vertex{-10}{10}\vertex{10}{10}\vertex{-10}{20}\vertex{10}{20}\vertex{0}{30}}

\lat{105}{-55}{L_{7,4}}{%
    \draw (0,0)--(-10,10)--(-10,20)--(0,30)--(10,20)--(0,10)--(0,0); 
    \draw (0,10)--(-10,20);
    \draw (0,10)--(0,30);
    \vertex{0}{0}\vertex{-10}{10}\vertex{0}{10}\vertex{-10}{20}\vertex{0}{20}\vertex{10}{20}\vertex{0}{30}}


\lat{0}{-110}{L_{7,4}^d}{%
    \draw (0,0)--(-10,10)--(-10,20)--(0,30)--(0,20)--(10,10)--(0,0);
    \draw (0,0)--(0,20);
    \draw (-10,10)--(0,20);
    \vertex{0}{0}\vertex{-10}{10}\vertex{0}{10}\vertex{10}{10}\vertex{0}{20}\vertex{-10}{20}\vertex{0}{30}}

\lat{25}{-110}{L_{7,5}}{%
    \draw (0,0)--(-10,10)--(-10,20)--(0,30)--(10,20)--(10,10)--(0,0); 
    \draw (-10,10)--(0,20)--(10,10);
    \draw (0,20)--(0,30);
    \vertex{0}{0}\vertex{-10}{10}\vertex{10}{10}\vertex{-10}{20}\vertex{0}{20}\vertex{10}{20}\vertex{0}{30}}

\lat{50}{-110}{K_{8,1}}{%
    \draw (0,0)--(-10,10)--(-10,20)--(0,30)--(10,20)--(10,10)--(0,0); 
    \draw (-10,10)--(0,20)--(0,30);
    \draw (0,0)--(0,10)--(10,20);
    \draw (-10,10)--(10,20);
    \vertex{0}{0}\vertex{0}{10}\vertex{-10}{10}\vertex{10}{10}\vertex{-10}{20}\vertex{0}{20}\vertex{10}{20}\vertex{0}{30}}

\lat{75}{-110}{K_{8,2}}{%
    \draw (0,0)--(-10,10)--(-10,20)--(0,30)--(10,20)--(10,10)--(0,0); 
    \draw (-10,10)--(0,20)--(10,10);
    \draw (0,20)--(0,30);
    \draw (0,0)--(0,10);
    \draw (-10,20)--(0,10)--(10,20);
    \vertex{0}{0}\vertex{0}{10}\vertex{-10}{10}\vertex{10}{10}\vertex{-10}{20}\vertex{0}{20}\vertex{10}{20}\vertex{0}{30}}

\lat{100}{-110}{L_{9,1}}{%
    \draw (0,0)--(-10,10)--(-10,20)--(0,30)--(10,20)--(10,10)--(0,0); 
    \draw (-10,10)--(0,20)--(10,10);
    \draw (0,20)--(0,30);
    \draw (0,0)--(0,10);
    \draw (-10,20)--(0,10)--(10,20);
    \draw (10,10)--(0,30);
    \vertex{0}{0}\vertex{0}{10}\vertex{-10}{10}\vertex{10}{10}\vertex{-10}{20}\vertex{0}{20}\vertex{5}{20}\vertex{10}{20}\vertex{0}{30}}


\lat{10}{-165}{K_{10,1}}{%
    \draw (0,0)--(-5,10)--(-15,20)--(0,30)--(15,20)--(5,10)--(0,0);
    \draw (0,0)--(-15,10)--(-5,20)--(0,30)--(5,20)--(15,10)--(0,0);
    \draw (-15,10)--(-15,20);
    \draw (15,10)--(15,20);
    \draw (-15,20)--(5,10);
    \draw (-5,20)--(15,10);
    \draw (-5,10)--(5,20);
    \vertex{0}{0}
    \vertex{-15}{10}\vertex{-5}{10}\vertex{5}{10}\vertex{15}{10}
    \vertex{-15}{20}\vertex{-5}{20}\vertex{5}{20}\vertex{15}{20}
    \vertex{0}{30}}

\lat{70}{-165}{L_{15,1}}{%
    \draw (0,30)--(-30,20); \draw (0,30)--(-20,20); \draw (0,30)--(-10,20); \draw (0,30)--(0,20); \draw (0,30)--(10,20); \draw (0,30)--(20,20); \draw (0,30)--(30,20);
    \draw (-30,20)--(-25,10); \draw (-30,20)--(-15,10); \draw (-30,20)--(-5,10);
    \draw (-20,20)--(-25,10); \draw (-20,20)--(25,10);
    \draw (-10,20)--(-25,10); \draw (-10,20)--(5,10); \draw (-10,20)--(15,10);
    \draw (0,20)--(-15,10); \draw (0,20)--(15,10);
    \draw (10,20)--(-15,10); \draw (10,20)--(5,10); \draw (10,20)--(25,10);
    \draw (20,20)--(-5,10); \draw (20,20)--(5,10);
    \draw (30,20)--(-5,10); \draw (30,20)--(15,10); \draw (30,20)--(25,10);
    \draw (0,0)--(-25,10); \draw (0,0)--(-15,10); \draw (0,0)--(-5,10); \draw (0,0)--(5,10); \draw (0,0)--(15,10); \draw (0,0)--(25,10);
    \vertex{0}{30}
    \vertex{-30}{20}\vertex{-20}{20}\vertex{-10}{20}\vertex{0}{20}\vertex{10}{20}\vertex{20}{20}\vertex{30}{20}
    \vertex{-25}{10}\vertex{-15}{10}\vertex{-5}{10}\vertex{5}{10}\vertex{15}{10}\vertex{25}{10}
    \vertex{0}{0}}

\end{tikzpicture}
}
\caption{Congruence lattices of $2$-valued unary algebras with at most four elements.}
\label{Fig:ConLat_on_4-1}
\end{figure}

\clearpage

\section{Classification of lattices with at most seven elements}\label{sect_seven-element-lattices}

\begin{notation*}\ 
\begin{itemize}
\item We use the notation $(abc\dots)\dots(xyz\dots)$ for the equivalence relation whose non-singleton blocks are $\{a,b,c,\dots\},\ldots,\{x,y,z,\dots\}$; singleton blocks are omitted. 
For example, $(12)(34)\in\Eq{[5]}$ denotes the equivalence relation with blocks $\{1,2\}$, $\{3,4\}$, and $\{5\}$.
\item The $k$-element chain (equivalently, the chain of length $k-1$) is denoted by $C_k$. Note that $C_1 = \mathbf{1}$ and $C_2 = \mathbf{2}$.
\item Each lattice is assigned a label of the form $m/n$, where $m$ is the number of elements of the lattice and $n$ is its index in the enumeration used below. 
\item The superscripts ${}^{(\mathrm{RD})}$, ${}^{(\mathrm{R})}$, and ${}^{(\mathrm{NR})}$ indicate, respectively, that the lattice is distributive and \tworepr{}, nondistributive and \tworepr{}, or not \tworepr{}.
\item For $a_1,\ldots,a_n\in[n]$, we denote by $\tau_{a_1a_2\cdots a_n}$ the transformation of $[n]$ 
defined by
\[
    i\mapsto a_i\qquad(i\in[n]).
\]
\end{itemize}
\end{notation*}

In the following, we examine lattices with at most $7$ elements regarding their \tworepry{}. 
(Let us mention that all but one lattices of size up to $7$ are known to be representable as congruence lattices of finite algebras \cite{DeMeo}; the only exception is the lattice $K_{7,4}$ on Figure~\ref{Fig:Lat7_part5}.)
For \tworepr{} lattices $L$, we identify the smallest base set $[m]$ ($m \in \mathbb{N}$) and a minimal set of transformations $M\subseteq T_2([m])$ such that $L\cong\Con{([m];M)}$ holds. 
The calculations were performed using the Mathematica software package. 
The statements below are used in cases of non-\tworepry{}.

As a consequence of Fact~\ref{fact_intervals-in-Eq(A)}, we can state the following.
\begin{proposition}\label{Prop:Forbidden_sublattices}
    Let $A$ be a finite set. 
    \begin{enumerate}[label=\rm(\alph*)]
        \item\label{item:chains-in-Prop:Forbidden_sublattices} No interval of $\Eq{A}$ is isomorphic to $C_k$ for any $k\geq 3$.
        \item\label{item:C2xC2-in-Prop:Forbidden_sublattices} There is no $u\in\Eq{A}$ such that the interval $[u;1_A]$ is isomorphic to $C_2\times C_2$.
        \item\label{item:Mk-in-Prop:Forbidden_sublattices} No interval of $\Eq{A}$ is isomorphic to $M_k$ for any $k\geq 4$.
    \end{enumerate}
\end{proposition}

\begin{corollary}\label{Cor:Chains}
    Let $L$ be a finite lattice. 
    If there exist distinct elements $u,v,w\in L\setminus\{0_L,1_L\}$ such that $[v,w]$ is a chain of length greater than $1$ and $u$ is a common relative complement of $v$ and $w$, then $L$ is not \tworepr{}.   
\end{corollary}

\begin{proof}
    Suppose that there exist distinct elements $u,v,w\in L\setminus\{0_L,1_L\}$ such that $[v,w]$ is a chain of length greater than $1$, and $u$ is a common relative complement of $v$ and $w$. 
    Then 
    \[
        [m^-(u,v,w);m^+(u,v,w)]=[v;w]\cong C_k
    \]
    for some integer $k\geq 3$. By Proposition~\ref{Prop:Forbidden_sublattices}\ref{item:chains-in-Prop:Forbidden_sublattices} and Corollary~\ref{cor_intabst}, 
    the lattice $L$ cannot be \tworepr{}.
\end{proof}

A lattice $L$ is \emph{semimodular} if, for all $a,b,c\in L$,
\[
    a\prec b\quad\Longrightarrow\quad a\vee c\preceq b\vee c,
\]
where $x\preceq y$ means that either $x=y$ or $x\prec y$.

\begin{theorem}[{\cite[Chapter~V, \S4.1, Theorem~404]{Gratzer11}}]\label{Th:Eq_is_semimodular}
    The lattice of equivalence relations on a finite set is semimodular.
\end{theorem}

\subsection{Lattices up to 3 elements}
Let $k\leq 3$ be a natural number. 
Up to isomorphism, there is exactly one lattice with $k$ elements, namely the $k$-element chain $C_k$. 
These lattices are distributive lattices; hence, they are \tworepr{}.
\medskip 

\RLatticeItem{RD}{$C_1\cong\Eq{[1]}=\Con{([1];\emptyset)}$.}
 
\RLatticeItem{RD}{$C_2\cong\Eq{[2]}=\Con{([2];\emptyset)}$.}
    
\RLatticeItem{RD}{$C_3\cong\{ 0_{[3]},(12),1_{[3]} \}\cong\Con{([3];\tau_{112})}$.}


\subsection{Lattices with 4 elements}
Up to isomorphism, there are two lattices with 4 elements: $C_4$ and $C_2\times C_2$. 
These lattices are distributive lattices; hence, they are \tworepr{}.
\medskip

\RLatticeItem{RD}{$C_2\times C_2\cong 
    \{ 0_{[4]}, (12)(34), (13)(24),1_{[4]} \}=\Con{([4];M)}$ for 
    $M=\{ \tau_{1133}, \tau_{1212}, \tau_{2244}, \tau_{3434} \}$.}

\RLatticeItem{RD}{$C_4\cong
	\{ 0_{[4]}, (34), (234),1_{[4]} \}=\Con{([4];M)}$ for 
	$M=\{ \tau_{1113}, \tau_{1121} \}$.}


\subsection{Lattices with 5 elements}\setcounter{LCard}{5}\setcounter{LType}{1}
Up to isomorphism, there are five lattices with 5 elements (see Figure~\ref{Fig:ConLat_on_5}). 
All these lattices are \tworepr.
\medskip

\RLatticeItem{RD}{$L_{5,1}\cong
    \{ 0_{[4]}, (34), (234), (12)(34), 1_{[4]} \}=\Con{([4];M)}$ 
    for $M=\{ \tau_{1222}, \tau_{3433} \}$.}
    
\RLatticeItem{RD}{$L_{5,1}^d\cong
    \{ 0_{[4]}, (23),(34), (234), 1_{[4]} \}=\Con{([4];M)}$ 
    for $M=\{ \tau_{2233}, \tau_{2444} \}$.}
    
\RLatticeItem{R}{$K_{5,1}\cong\Con{([3];\emptyset)}=\Eq{[3]}$.}
    
\RLatticeItem{R}{$K_{5,2}\cong
    \{ 0_{[4]}, (34), (12)(34), (13)(24),1_{[4]} \}=\Con{([4];M)}$ 
    for $M=\{ \tau_{1133}, \tau_{2244}, \tau_{3434} \}$.}
    
\RLatticeItem{RD}{$K_{5,3}\cong
    \{ 0_{[5]}, (45), (345), (2345), 1_{[5]} \}=\Con{([5];M)}$ 
    for $M=\{ \tau_{11114}, \tau_{11131}, \tau_{11211} \}$.}
    

\subsection{Lattices with 6 elements}\setcounter{LCard}{6}\setcounter{LType}{1}
Up to isomorphism, there are fifteen lattices with 6 elements. 
\medskip

\NRLatticeItem{NR}{$L_{6,1}$: Let $L=L_{6,1}$.
For $u,v,w,s\in L$ as indicated in Figure~\ref{Fig:Lat_6a} on page~\pageref{Fig:Lat_6a}, we have $[m^-(u,v,w);m^+(u,v,w)] = [s;1] \cong C_2\times C_2$. It follows from Proposition~\ref{Prop:Forbidden_sublattices}\ref{item:C2xC2-in-Prop:Forbidden_sublattices} and Corollary~\ref{cor_intabst} that $L$ is not \tworepr{}.}

\RLatticeItem{R}{$L_{6,1}^d\cong
    \{ 0_{[4]}, (14),(23), (12)(34),(14)(23), 1_{[4]} \}=\Con{([4];M)}$
    for $M=\{ \tau_{1221}, \tau_{3443} \}$.} 

\RLatticeItem{R}{$L_{6,2}\cong
    \{ 0_{[4]}, (34), (134), (234), (12)(34), 1_{[4]}\}=\Con{([4];M)}$ 
    for $M=\{ \tau_{3344}, \tau_{3433} \}$.} 

\RLatticeItem{R}{$L_{6,2}^d\cong
    \{ 0_{[4]}, (23),(24),(34), (234), 1_{[4]} \}=\Con{([4];M)}$ 
    for $M=\{ \tau_{2333}, \tau_{2444} \}$.}
        
\RLatticeItem{R}{$L_{6,3}
    \cong\{ 0_{[5]}, (12), (123), (123)(45),(124)(35), 1_{[5]} \}
    =\Con{([5];M)}$
    for $M=\{ \tau_{11211}, \tau_{11313}, \tau_{11144}, \tau_{33355} \}$.} 

\RLatticeItem{R}{$L_{6,3}^d\cong
    \{ 0_{[5]}, (12), (12)(34),(13)(24), (1234), 1_{[5]} \}=\Con{([5];M)}$ 
    for $M=\{ \tau_{11112}, \tau_{12121}, \tau_{11331}, \tau_{22444} \}$.} 

\RLatticeItem{RD}{$L_{6,4}\cong
    \{ 0_{[5]}, (12), (123), (1234), (1235), 1_{[5]} \}=\Con{([5];M)}$ 
    for $M=\{ \tau_{11115}, \tau_{11133}, \tau_{11211} \}$.} 

\RLatticeItem{RD}{$L_{6,4}^d\cong
    \{ 0_{[5]}, (12),(13), (123), (1234), 1_{[5]} \}=\Con{([5];M)}$ 
    for $M=\{ \tau_{11114}, \tau_{11311}, \tau_{22232} \}$.} 
    
\NRLatticeItem{NR}{$K_{6,1}$:  Let $L=K_{6,1}$. 
    The subset $\{u,v,w\}$ of $L$, where $u,v,w\in L$ are as indicated in 
    Figure~\ref{Fig:Lat_6b} on page~\pageref{Fig:Lat_6b}, 
    generates a $(0,1)$-sublattice that is isomorphic to $M_3$. 
    By Proposition~\ref{prop_median}, the lattice $L$ is not \tworepr{} since $|L|=6$ is not a Bell number.}

\NRLatticeItem{NR}{$K_{6,2}$:  Let $L=K_{6,2}$. 
    The subset $\{u,v,w\}$ of $L$, where $u,v,w\in L$ are as indicated in 
    Figure~\ref{Fig:Lat_6b} on page~\pageref{Fig:Lat_6b}, 
    generates a $(0,1)$-sublattice that is isomorphic to $M_3$. 
    By Proposition~\ref{prop_median}, the lattice $L$ is not \tworepr{} since $|L|=6$ is not a Bell number.}

\RLatticeItem{R}{$K_{6,3}$: 
    \[ 
        K_{6,3}\cong
        \{ 0_{[6]}, (15)(36), (23)(456), (15)(24)(36), (123)(456), 1_{[6]}\}=\Con{([6];M)}
    \]
    for $M=\{ \tau_{111555}, \tau_{232322}, \tau_{333666}, \tau_{445445},\tau_{446446} \}$.}

\RLatticeItem{RD}{$K_{6,4}\cong
    \{ 0_{[4]}, (23),(34), (12)(34), (234), 1_{[4]}\}=\Con{([4];M)}$ 
    for $M=\{ \tau_{1222}, \tau_{3443} \}$.} 

\NRLatticeItem{NR}{$K_{6,5}$: Let $L=K_{6,5}$. 
	Applying Corollary~\ref{Cor:Chains} to the elements $u,v,w\in L$ indicated in Figure~\ref{Fig:Lat_6c} on page~\pageref{Fig:Lat_6c} shows that the lattice $L$ is not \tworepr{}.}
        
\RLatticeItem{RD}{$K_{6,6}\cong
        \{ 0_{[5]}, (12), (123),(124), (1234), 1_{[5]}\}=\Con{([5];M)}$ 
        for $M=\{ \tau_{11113}, \tau_{11114}, \tau_{11141}, \tau_{11221} \}$.}

\RLatticeItem{RD}{$K_{6,7}\cong
    \{ 0_{[6]}, (56), (456), (3456), (23456), 1_{[6]} \}=\Con{([6];M)}$  
    for $M=\{ \tau_{111115}, \tau_{111141}, \tau_{111311}, \tau_{112111} \}$.}


\subsection{Lattices with 7 elements}\setcounter{LCard}{7}\setcounter{LType}{1}
Up to isomorphism, there are fifty-three lattices with 7 elements. 
\medskip

\NRLatticeItem{NR}{$L_{7,1}$: Let $L=L_{7,1}$. 
    The subset $\{u,v,w\}$ of $L$, where $u,v,w\in L$ are as indicated in 
    Figure~\ref{Fig:Lat7_part1} on page~\pageref{Fig:Lat7_part1}, 
    generates a $(0,1)$-sublattice that is isomorphic to $M_3$. 
    By Proposition~\ref{prop_median}, the lattice $L$ is not \tworepr{} since $|L|=7$ is not a Bell number.}  
        
\NRLatticeItem{NR}{$L_{7,1}^d$: Let $L=L_{7,1}^d$. 
    The subset $\{u,v,w\}$ of $L$, where $u,v,w\in L$ are as indicated in 
    Figure~\ref{Fig:Lat7_part1} on page~\pageref{Fig:Lat7_part1}, 
    generates a $(0,1)$-sublattice that is isomorphic to $M_3$. 
    By Proposition~\ref{prop_median}, the lattice $L$ is not \tworepr{} since $|L|=7$ is not a Bell number.}
        
\NRLatticeItem{NR}{$L_{7,2}$: Let $L=L_{7,2}$, 
    and suppose that $\phi\colon L\to\Eq{A}$ is a \twogood{} for a finite set $A$. 
    For the DNF $t^-=(x\wedge y)\vee(x\wedge z)\in\FL{3}$ and the corresponding CNF $t^+=x\wedge(y\vee z)\in\FL{3}$, we have
    \[
        [t^-(u,v,w);t^+(u,v,w)]
        = [0_L;u]
    \]
	where $u,v,w\in L$ are as indicated in Figure~\ref{Fig:Lat7_part1} on page~\pageref{Fig:Lat7_part1}.
	By Corollary~\ref{cor_intabst}, $\varphi([0_L;u]_L)=[0_A;\varphi(u)]_{\Eq{A}}$, hence $\phi(u)\wedge\phi(s)=0_A\prec\phi(u)$. 
    By Theorem~\ref{Th:Eq_is_semimodular}, $\Eq{A}$ is semimodular, thus it follows that $\phi(s)\prec\phi(u)\vee\phi(s)=1_A$,  
    contradicting $\phi(s)<\phi(v)< 1_A$.}
        
\NRLatticeItem{NR}{$L_{7,2}^d$: Let $L=L_{7,2}^d$, 
    and suppose that $\phi\colon L\to\Eq{A}$ is a \twogood{} for a finite set $A$. 
    For the DNF $t^-=x\vee(y\wedge z)\in\FL{3}$ and the corresponding CNF $t^+=(x\vee y)\wedge(x\vee z)\in\FL{3}$,
    we have  
    \begin{align*}
        [t^-(u,v,w);t^+(u,v,w)]_L &= [u;1_L]_L,\\
        [m^-(u,v,w);m^+(u,v,w)]_L &= [0_L;s]_L,
    \end{align*}
    where $u,v,w,s\in L$ are as indicated in Figure~\ref{Fig:Lat7_part1} on
    page~\pageref{Fig:Lat7_part1}.
    By Corollary~\ref{cor_intabst}, $\varphi([u;1_L]_L)=[\varphi(u);1_A]_{\Eq{A}}$ and $\varphi([0_L;s]_L)=[0_A;\varphi(s)]_{\Eq{A}}$.
    The first one of these two equalities implies that $\phi(u)\prec 1_A$, hence the equivalence relation $\phi(u)$ has exactly two blocks.
    The second equality implies that $0_A\prec\phi(u),\phi(v),\phi(w)$ and these three atoms have pairwise join $\phi(s)$, hence the equivalence relation $\phi(s)$ has a unique non-singleton block $B$ of size $3$. 
    However, $\phi(u)\wedge\phi(s)=0_A$ implies that the three elements of $B$ belong 
    to distinct blocks of $\phi(u)$, a contradiction.
    Therefore, $L$ is not \tworepr{}.}
    
\NRLatticeItem{NR}{$L_{7,3}$: Let $L=L_{7,3}$. 
	Consider the term $t=(x\wedge y)\vee(x\wedge z)\in\FL{3}$. For $u,v,w\in L$ as indicated in Figure~\ref{Fig:Lat7_part1} on page~\pageref{Fig:Lat7_part1}, we have $[t^-(u,v,w);t^+(u,v,w)] = [0;u] \cong C_3$. It follows from Proposition~\ref{Prop:Forbidden_sublattices}\ref{item:chains-in-Prop:Forbidden_sublattices} and Corollary~\ref{cor_intabst} that $L$ is not \tworepr{}.}
        
\NRLatticeItem{NR}{$L_{7,3}^d$: see Example~\ref{example_L73d}.} 
    
\RLatticeItem{R}{$L_{7,4}\cong
    \{ 0_{[4]}, (12),(34), (12)(34),(134),(234), 1_{[4]}\}=\Con{([4];M)}$ 
    for $M=\{ \tau_{3344} \}$.}
        
\RLatticeItem{R}{$L_{7,4}^d\cong
    \{ 0_{[4]}, (12),(13),(23), (123),(124), 1_{[4]}\}=\Con{([4];M)}$ 
    for $M=\{ \tau_{1112},\tau_{1114} \}$.}
    
\RLatticeItem{R}{$L_{7,5}$:
    \begin{equation*} 
        L_{7,5}
        \cong\{ 0_{[4]}, (12),(13), (12)(34),(13)(24), (123), 1_{[4]}\}
        =\Con{([4];M)}
    \end{equation*} 
    for $M=\{ \tau_{1133},\tau_{1212} \}$.}

\RLatticeItem{R}{$L_{7,5}^d\cong
    \{ 0_{[5]}, (23), (12)(45),(14)(35), (123)(45), (14)(235), 1_{[5]} \}=\Con{([5];M)}$ 
    for $M=\{ \tau_{11144},\tau_{12212},\tau_{33355},\tau_{45545} \}$.}
    
\NRLatticeItem{NR}{$L_{7,6}$: Let $L=L_{7,6}$. For $u\in L$ as indicated in Figure~\ref{Fig:Lat7_part2} on page~\pageref{Fig:Lat7_part2}, the filter $[u;1]\cong K_{6,1} \cong M_4$ is not \tworepr{} (see case 6/9), hence it follows from Proposition~\ref{prop_filters} that $L$ cannot be \tworepr{} either.}

\NRLatticeItem{NR}{$L_{7,6}^d$: Let $L=L_{7,6}^d$. For $u,v,w,s\in L$ as indicated in Figure~\ref{Fig:Lat7_part2} on page~\pageref{Fig:Lat7_part2}, we have $[m^-(u,v,w);m^+(u,v,w)] = [0;s] \cong M_4$. It follows from Proposition~\ref{Prop:Forbidden_sublattices}\ref{item:Mk-in-Prop:Forbidden_sublattices} and Corollary~\ref{cor_intabst} that $L$ is not \tworepr{}.}

\NRLatticeItem{NR}{$L_{7,7}$: Let $L=L_{7,7}$.
 	Applying Corollary~\ref{Cor:Chains} to the elements $u,v,w\in L$ indicated in Figure~\ref{Fig:Lat7_part2} on page~\pageref{Fig:Lat7_part2} shows that the lattice $L$ is not \tworepr{}.}

\NRLatticeItem{NR}{$L_{7,7}^d$: Let $L=L_{7,7}^d$.
	Applying Corollary~\ref{Cor:Chains} to the elements $u,v,w\in L$ indicated in Figure~\ref{Fig:Lat7_part2} on page~\pageref{Fig:Lat7_part2} shows that the lattice $L$ is not \tworepr{}.}

\NRLatticeItem{NR}{$L_{7,8}$: Let $L=L_{7,8}$. 
	Applying Corollary~\ref{Cor:Chains} to the elements $u,v,w\in L$ indicated in Figure~\ref{Fig:Lat7_part2} on page~\pageref{Fig:Lat7_part2} shows that the lattice $L$ is not \tworepr{}.}

\NRLatticeItem{NR}{$L_{7,8}^d$: Let $L=L_{7,8}^d$. 
	Applying Corollary~\ref{Cor:Chains} to the elements $u,v,w\in L$ indicated in Figure~\ref{Fig:Lat7_part2} on page~\pageref{Fig:Lat7_part2} shows that the lattice $L$ is not \tworepr{}.}

\RLatticeItem{R}{$L_{7,9}\cong\{0_{[5]}, (13), (12)(34),(13)(24), (125)(34), (1234), 1_{[5]} \}=\Con([5];M)$ 
    for $M=\{ \tau_{11331}, \tau_{12121}, \tau_{34344} \}$.} 

\RLatticeItem{R}{$L_{7,9}^d
    \cong\{0_{[5]}, (15), (12)(34),(135), (125)(34), (135)(24), 1_{[5]} \}
    =\Con{([5];M)}$ for $M=\{ \tau_{11115}, \tau_{12121}, \tau_{33553}, \tau_{34343} \}$.}

\RLatticeItem{R}{$L_{7,10}
    \cong\{ 0_{[6]}, (14), (123), (1234), (1234)(56),(1235)(46), 1_{[6]}\}=\Con{([6];M)}$ 
    for $M=\{\tau_{111155},\tau_{111414},\tau_{112111},\tau_{131113},\tau_{444466}\}$.}

\RLatticeItem{R}{$L_{7,10}^d
    \cong\{ 0_{[5]}, (12), (12)(34),(13)(24), (12)(345), (1234), 1_{[5]}\}=\Con{([5];M)}$ 
    for $M=\{ \tau_{11333},\tau_{12121},\tau_{22444},\tau_{33334}\}$.}

\NRLatticeItem{NR}{$L_{7,11}$: Let $L=L_{7,11}$. For $u\in L$ as indicated in Figure~\ref{Fig:Lat7_part3} on page~\pageref{Fig:Lat7_part3}, the filter $[u;1]\cong K_{6,2}$ is not \tworepr{} (see case 6/10), hence it follows from Proposition~\ref{prop_filters} that $L$ cannot be \tworepr{} either.}

\NRLatticeItem{NR}{$L_{7,11}^d$: Let $L=L_{7,11}^d$. Consider the term $t=x\vee(y\wedge z)\in\FL{3}$. For $u,v,w,s\in L$ as indicated in Figure~\ref{Fig:Lat7_part3} on page~\pageref{Fig:Lat7_part3}, we have $[t^-(u,v,w);t^+(u,v,w)] = [u;s] \cong C_3$. It follows from Proposition~\ref{Prop:Forbidden_sublattices}\ref{item:Mk-in-Prop:Forbidden_sublattices} and Corollary~\ref{cor_intabst} that $L$ is not \tworepr{}.}
    
\RLatticeItem{R}{$L_{7,12}\cong 
    \{ 0_{[5]}, (12), (12)(34),(125), (124)(35), (125)(34), 1_{[5]} \}=\Con([5];M)$
    for $M=\{ \tau_{11122},\tau_{11441},\tau_{33553} \}$.}
    
\NRLatticeItem{NR}{$L_{7,12}^d$: Let $L=L_{7,12}^d$. For $u,v,w\in L$ as indicated in Figure~\ref{Fig:Lat7_part3} on page~\pageref{Fig:Lat7_part3}, we have $[m^-(u,v,w);m^+(u,v,w)] = [v\wedge w;v\vee w] \cong C_2 \times C_2$. It follows from Proposition~\ref{Prop:Forbidden_sublattices}\ref{item:C2xC2-in-Prop:Forbidden_sublattices} and Corollary~\ref{cor_intabst} that $L$ is not \tworepr{}.}
    
\NRLatticeItem{NR}{$L_{7,13}$: Let $L=L_{7,13}$. For $u\in L$ as indicated in Figure~\ref{Fig:Lat7_part3} on page~\pageref{Fig:Lat7_part3}, the filter $[u;1]\cong L_{6,1}$ is not \tworepr{} (see case 6/1), hence it follows from Proposition~\ref{prop_filters} that $L$ cannot be \tworepr{} either.}
    
\RLatticeItem{R}{$L_{7,13}^d$: 
    \[
        L_{7,13}^d
        \cong\{ 0_{[5]}, (12),(34), (12)(34),(13)(24), (1234), 1_{[5]} \}
        =\Con([5];M)
    \]
    for $M=\{ \tau_{11112},\tau_{11331},\tau_{22444} \}$.}
    
\RLatticeItem{R}{$L_{7,14}$: 
    \begin{align*}
        L_{7,14}
        &\cong \{ 0_{[7]}, (12), (125)(37), (125)(37)(46), (1234)(56), (1234)(567), 1_{[7]} \}\\
        &= \Con([7];M)
    \end{align*}
    for $M=\{ \tau_{2111211},\tau_{3332323},\tau_{4434443},\tau_{5555222}, \tau_{6655655}, \tau_{7777333} \}$.}
        
\RLatticeItem{R}{$L_{7,14}^d$: 
    \begin{align*}
        L_{7,14}^d
        &\cong \{ 0_{[7]}, (14)(27), (14)(27)(35), (123)(45), (123)(457), (123457), 1_{[7]} \}\\
        &= \Con([7];M)
    \end{align*}
    for $M=\{ \tau_{1114414},\tau_{1121221},\tau_{1311133}, \tau_{2227777}, \tau_{4454544} \}$.}
    
\RLatticeItem{RD}{$L_{7,15}$: 
    \begin{align*}
        L_{7,15}
        &\cong \{ 0_{[5]}, (12), (12)(35), (125),(1235), (1245), 1_{[5]} \}
        = \Con([5];M)
    \end{align*}
    for $M=\{ \tau_{11141},\tau_{11555}, \tau_{22221} \}$.}
    
\RLatticeItem{RD}{$L_{7,15}^d$: 
    \[ 
        L_{7,15}^d\cong 
        \{ 0_{[5]}, (14),(23), (14)(23), (134), (1234), 1_{[5]} \}=\Con([5];M)
    \]
    for $M=\{ \tau_{11112},\tau_{11141},\tau_{34433} \}$.}
    
\RLatticeItem{R}{$L_{7,16}$: 
    \[
        L_{7,16}\cong 
        \{ 0_{[5]}, (12), (124), (1234),(1245), (124)(35), 1_{[5]} \}=\Con([5];M)
    \]
    for $M=\{ \tau_{11144},\tau_{22442} \}$.}
    
\RLatticeItem{R}{$L_{7,16}^d\cong 
    \{ 0_{[5]}, (12),(13),(23), (123), (1235), 1_{[5]} \}=\Con([5];M)$
    for $M=\{ \tau_{11112},\tau_{11113},\tau_{11151} \}$.}
    
\NRLatticeItem{NR}{$L_{7,17}$: Let $L=L_{7,17}$. For $u\in L$ as indicated in Figure~\ref{Fig:Lat7_part4} on page~\pageref{Fig:Lat7_part4}, the filter $[u;1]\cong K_{6,5}$ is not \tworepr{} (see case 6/13), hence it follows from Proposition~\ref{prop_filters} that $L$ cannot be \tworepr{} either.}
    
\NRLatticeItem{NR}{$L_{7,17}^d$: Let $L=L_{7,17}^d$.
	Applying Corollary~\ref{Cor:Chains} to the elements $u,v,w\in L$ indicated in Figure~\ref{Fig:Lat7_part4} on page~\pageref{Fig:Lat7_part4} shows that the lattice $L$ is not \tworepr{}.}

\RLatticeItem{R}{$L_{7,18}\cong 
    \{ 0_{[6]}, (12), (123), (1234), (1234)(56), (1235)(46), 1_{[6]} \}=\Con([6];M)$
    for $M=\{ \tau_{111155},\tau_{111414}, \tau_{112111}, \tau_{113311}, \tau_{444466} \}$.}
        
\RLatticeItem{R}{$L_{7,18}^d$: 
    \[
        L_{7,18}^d\cong 
        \{ 0_{[6]}, (13), (12)(34),(13)(24), (1234), (1234)(56), 1_{[6]} \}=\Con([6];M)
    \]
    for $M=\{ \tau_{113313}, \tau_{121211},\tau_{343434}, \tau_{555566} \}$.}
    
\RLatticeItem{RD}{$L_{7,19}$: 
    \begin{align*}
        L_{7,19}
        &\cong \{ 0_{[6]}, (46), (15)(46),(15)(23)(46), (15)(346), (15)(2346), 1_{[6]} \}\\
        &= \Con([6];M)
    \end{align*}
    for $M=\{ \tau_{115151}, \tau_{232222},  \tau_{244424}, \tau_{464466} \}$.}
        
\RLatticeItem{RD}{$L_{7,19}^d$: 
    \[
        L_{7,19}^d\cong\{ 0_{[6]}, (23), (14)(23), (235), (14)(235), (14)(2356), 1_{[6]} \}=\Con([6];M)
    \]
    for $M=\{ \tau_{111414}, \tau_{222225}, \tau_{233323},\tau_{266266} \}$.}
        
\RLatticeItem{RD}{$L_{7,20}$: 
    \[
        L_{7,20}\cong\{ 0_{[6]}, (16), (16)(25), (146)(25), (12456), (146)(235), 1_{[6]} \}=\Con([6];M)
    \]
    for $M=\{ \tau_{161116}, \tau_{255552}, \tau_{455454}, \tau_{464464} \}$.}
    
\RLatticeItem{RD}{$L_{7,20}^d$: 
    \[
        L_{7,20}^d\cong\{ 0_{[6]}, (25),(45), (245), (2345), (12345), 1_{[6]} \}=\Con([6];M)
    \]
    for $M=\{ \tau_{111112}, \tau_{222552}, \tau_{223222}, \tau_{343443} \}$.}
    
\NRLatticeItem{NR}{$K_{7,1}$:  Let $L=K_{7,1}$. 
    The subset $\{u,v,w\}$ of $L$, where $u,v,w\in L$ are as indicated in 
    Figure~\ref{Fig:Lat7_part5} on page~\pageref{Fig:Lat7_part5}, 
    generates a $(0,1)$-sublattice that is isomorphic to $M_3$. 
    By Proposition~\ref{prop_median}, the lattice $L$ is not \tworepr{} since $|L|=7$ is not a Bell number.}

\NRLatticeItem{NR}{$K_{7,2}$: Let $L=K_{7,2}$. 
    The subset $\{u,v,w\}$ of $L$, where $u,v,w\in L$ are as indicated in 
    Figure~\ref{Fig:Lat7_part5} on page~\pageref{Fig:Lat7_part5}, 
    generates a $(0,1)$-sublattice that is isomorphic to $M_3$. 
    By Proposition~\ref{prop_median}, the lattice $L$ is not \tworepr{} since $|L|=7$ is not a Bell number.}
    
\NRLatticeItem{NR}{$K_{7,3}$: Let $L=K_{7,3}$. 
    The subset $\{u,v,w\}$ of $L$, where $u,v,w\in L$ are as indicated in 
    Figure~\ref{Fig:Lat7_part5} on page~\pageref{Fig:Lat7_part5}, 
    generates a $(0,1)$-sublattice that is isomorphic to $M_3$. 
    By Proposition~\ref{prop_median}, the lattice $L$ is not \tworepr{} since $|L|=7$ is not a Bell number.}
    
\NRLatticeItem{NR}{$K_{7,4}$: Let $L=K_{7,4}$. 
    The subset $\{u,v,w\}$ of $L$, where $u,v,w\in L$ are as indicated in 
    Figure~\ref{Fig:Lat7_part5} on page~\pageref{Fig:Lat7_part5}, 
    generates a $(0,1)$-sublattice that is isomorphic to $M_3$. 
    By Proposition~\ref{prop_median}, the lattice $L$ is not \tworepr{} since $|L|=7$ is not a Bell number.}
    
\NRLatticeItem{NR}{$K_{7,5}$: Let $L=K_{7,5}$. 
    The subset $\{u,v,w\}$ of $L$, where $u,v,w\in L$ are as indicated in 
    Figure~\ref{Fig:Lat7_part5} on page~\pageref{Fig:Lat7_part5}, 
    generates a $(0,1)$-sublattice that is isomorphic to $M_3$. 
    By Proposition~\ref{prop_median}, the lattice $L$ is not \tworepr{} since $|L|=7$ is not a Bell number.}
    
\RLatticeItem{R}{$K_{7,6}$: 
    \begin{align*}
        K_{7,6}
        &\cong \{ 0_{[6]}, (34)(56), (13)(26)(45), (156)(34),(234)(56), (156)(234), 1_{[6]} \}\\
        &= \Con([6];M)
    \end{align*}
    for $M=\{ \tau_{133311},\tau_{266622},\tau_{333443}, \tau_{455544}, \tau_{565666} \}$.}
    
\NRLatticeItem{NR}{$K_{7,7}$: Let $L=K_{7,7}$. 
	Applying Corollary~\ref{Cor:Chains} to the elements $u,v,w\in L$ indicated in Figure~\ref{Fig:Lat7_part5} on page~\pageref{Fig:Lat7_part5} shows that the lattice $L$ is not \tworepr{}.}

\NRLatticeItem{NR}{$K_{7,8}$: Let $L=K_{7,8}$. For $u,v,w\in L$ as indicated in Figure~\ref{Fig:Lat7_part5} on page~\pageref{Fig:Lat7_part5}, we have $[m^-(u,v,w);m^+(u,v,w)] = [v\wedge w;u\vee v] \cong C_3$. It follows from Proposition~\ref{Prop:Forbidden_sublattices}\ref{item:chains-in-Prop:Forbidden_sublattices} and Corollary~\ref{cor_intabst} that $L$ is not \tworepr{}.}
    
\RLatticeItem{RD}{$K_{7,9}$: 
    \[
        K_{7,9}\cong\{ 0_{[5]}, (12),(14), (124), (124)(35), (1245), 1_{[5]} \}=\Con([5];M)
    \]
    for $M=\{ \tau_{11112},\tau_{11144},\tau_{11515} \}$.}
        
\RLatticeItem{R}{$K_{7,10}$: 
    \[
        K_{7,10}\cong\{ 0_{[5]}, (14), (14)(23), (124), (134), (1234), 1_{[5]} \}=\Con([5];M)
    \]
    for $M=\{ \tau_{11112},\tau_{11113},\tau_{11411}, \tau_{14111} \}$.}
    
\NRLatticeItem{NR}{$K_{7,11}$: Let $L=K_{7,11}$. 
	Applying Corollary~\ref{Cor:Chains} to the elements $u,v,w\in L$ indicated in Figure~\ref{Fig:Lat7_part5} on page~\pageref{Fig:Lat7_part5} shows that the lattice $L$ is not \tworepr{}.}
    
\RLatticeItem{R}{$K_{7,12}\cong 
    \{ 0_{[6]}, (12), (124), (123)(45), (124)(35), (12345), 1_{[6]} \}=\Con([6];M)$
    for $M=\{ \tau_{111121}, \tau_{111441}, \tau_{113131}, \tau_{222444}, \tau_{445455} \}$.}
        
\RLatticeItem{RD}{$K_{7,13}$: 
    \[
        K_{7,13}\cong
        \{ 0_{[7]}, (12), (123), (1234), (12345), (123456), 1_{[7]} \}=\Con{([7];M)}
    \]
    for $M=\{ \tau_{1111116}, \tau_{1111151}, \tau_{1111411}, \tau_{1113111}, \tau_{1121111} \}$.}


\clearpage

\subsection{Figures}

\ 

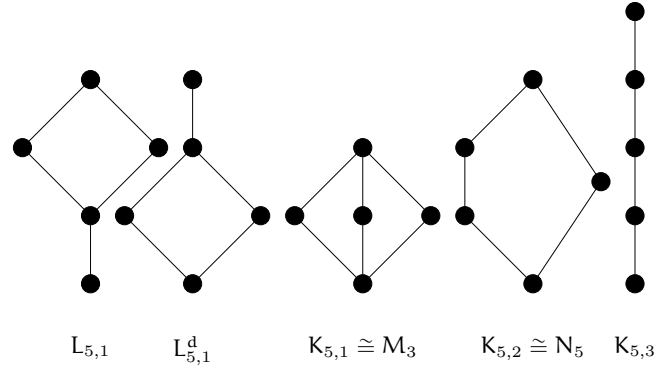
\begin{figure}[ht]
\centering
\scalebox{1.0}{
\begin{tikzpicture}[x=0.9mm,y=0.9mm, line width=0.35pt,
  lab/.style={draw=none, fill=none, below=3mm, font=\small}]

\newcommand{\vertex}[2]{\filldraw[black] (#1,#2) circle (1.2mm);}
\newcommand{\vertexL}[5]{\filldraw[black] (#1,#2) circle (1.2mm); \node at (#4,#5) {$#3$};}
\newcommand{\lat}[4]{%
  \begin{scope}[shift={(#1,#2)}]
    #4
    \node[lab] at (0,-3) {$#3$};
  \end{scope}
}


\lat{0}{0}{L_{5,1}}{%
    \draw (0,0)--(0,10)--(-10,20)--(0,30)--(10,20)--(0,10);
    \vertex{0}{0}\vertex{0}{10}\vertex{-10}{20}\vertex{0}{30}\vertex{10}{20}}

\lat{15}{0}{L_{5,1}^d}{%
    \draw (0,0)--(-10,10)--(0,20)--(0,30); 
    \draw (0,0)--( 10,10)--(0,20);
    \vertex{0}{0}\vertex{-10}{10}\vertex{10}{10}\vertex{0}{20}\vertex{0}{30}}

\lat{40}{0}{K_{5,1}\cong M_3}{%
    \draw (0,0)--(-10,10)--(0,20); 
    \draw (0,0)--(  0,10)--(0,20);
    \draw (0,0)--( 10,10)--(0,20); 
    \vertex{0}{0}\vertex{-10}{10}\vertex{0}{10}\vertex{10}{10}\vertex{0}{20}}

\lat{65}{0}{K_{5,2}\cong N_5}{%
    \draw (0,0)--(-10,10)--(-10,20)--(0,30)--(10,15)--(0,0);
    \vertex{0}{0}\vertex{-10}{10}\vertex{-10}{20}\vertex{0}{30}\vertex{10}{15}}

\lat{80}{0}{K_{5,3}}{%
    \draw (0,0)--(0,40);
    \vertex{0}{0}\vertex{0}{10}\vertex{0}{20}\vertex{0}{30}\vertex{0}{40}}

\end{tikzpicture}
}
\caption{Non-isomorphic lattices with $5$ elements}
\label{Fig:ConLat_on_5}
\end{figure}

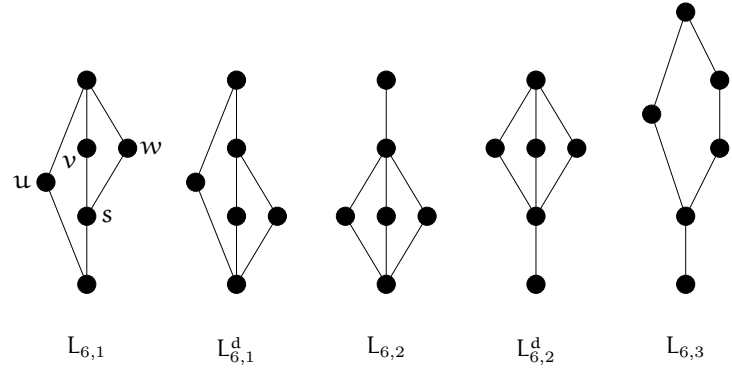
\begin{figure}[ht]
\centering
\scalebox{1.0}{
\begin{tikzpicture}[x=0.9mm,y=0.9mm, line width=0.35pt,
  lab/.style={draw=none, fill=none, below=3mm, font=\small}]

\newcommand{\vertex}[2]{\filldraw[black] (#1,#2) circle (1.2mm);}
\newcommand{\vertexL}[5]{\filldraw[black] (#1,#2) circle (1.2mm); \node at (#4,#5) {$#3$};}
\newcommand{\lat}[4]{%
  \begin{scope}[shift={(#1,#2)}]
    #4
    \node[lab] at (0,-3) {$#3$};
  \end{scope}
}


\lat{0}{0}{L_{6,1}}{
  \draw (0,0)--(0,10)--(0,20)--(0,30);
  \draw (0,0)--(-6,15)--(0,30);
  \draw (0,10)--(6,20)--(0,30);
  \vertex{0}{0}
  \vertexL{0}{10}{s}{3}{10}
  \vertexL{0}{20}{v}{-2.5}{18}
  \vertex{0}{30}\vertexL{-6}{15}{u}{-9.5}{15}\vertexL{6}{20}{w}{9.5}{20}
}

\lat{22}{0}{L_{6,1}^d}{
  \draw (0,30)--(0,20)--(0,10)--(0,0);
  \draw (0,30)--(-6,15)--(0,0);
  \draw (0,20)--(6,10)--(0,0);
  \vertex{0}{30}\vertex{0}{20}\vertex{0}{10}\vertex{0}{0}\vertex{-6}{15}\vertex{6}{10}
}

\lat{44}{0}{L_{6,2}}{
  \draw (0,0)--(-6,10);
  \draw (0,0)--(0,10);
  \draw (0,0)--(6,10);
  \draw (-6,10)--(0,20);
  \draw (0,10)--(0,20);
  \draw (6,10)--(0,20);
  \vertex{0}{0}\vertex{-6}{10}\vertex{0}{10}\vertex{6}{10}\vertex{0}{20}\vertex{0}{30}
  \draw (0,20)--(0,30);
}

\lat{66}{0}{L_{6,2}^d}{
  \draw (0,30)--(-6,20);
  \draw (0,30)--(0,20);
  \draw (0,30)--(6,20);
  \draw (-6,20)--(0,10);
  \draw (0,20)--(0,10);
  \draw (6,20)--(0,10);
  \draw (0,10)--(0,0);
  \vertex{0}{30}\vertex{-6}{20}\vertex{0}{20}\vertex{6}{20}\vertex{0}{10}\vertex{0}{0}
}

\lat{88}{0}{L_{6,3}}{
  \draw (0,0)--(0,10);
  \draw (0,10)--(-5,25)--(0,40);
  \draw (0,10)--(5,20)--(5,30)--(0,40);
  \vertex{0}{0}\vertex{0}{10}\vertex{-5}{25}\vertex{5}{20}\vertex{5}{30}\vertex{0}{40}
}

\end{tikzpicture}
}
\caption{Non-isomorphic lattices with $6$ elements (Part I)}
\label{Fig:Lat_6a}
\end{figure}


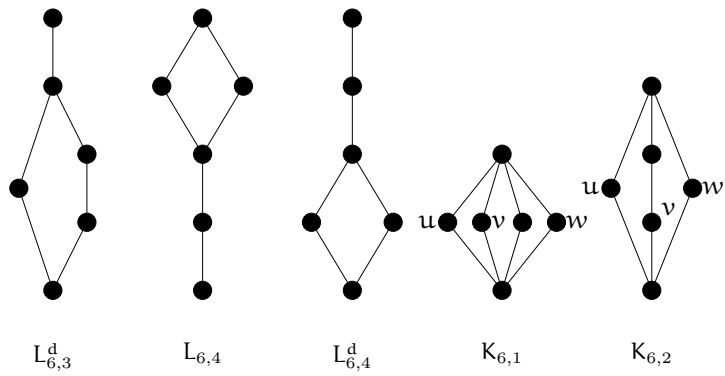
\begin{figure}[ht]
\centering
\scalebox{1.0}{
\begin{tikzpicture}[x=0.9mm,y=0.9mm, line width=0.35pt,
  lab/.style={draw=none, fill=none, below=3mm, font=\small}]

\newcommand{\vertex}[2]{\filldraw[black] (#1,#2) circle (1.2mm);}
\newcommand{\vertexL}[5]{\filldraw[black] (#1,#2) circle (1.2mm); \node at (#4,#5) {$#3$};}
\newcommand{\lat}[4]{%
  \begin{scope}[shift={(#1,#2)}]
    #4
    \node[lab] at (0,-3) {$#3$};
  \end{scope}
}


\lat{0}{-62}{L_{6,3}^d}{
  \draw (0,40)--(0,30);
  \draw (0,30)--(-5,15)--(0,0);
  \draw (0,30)--(5,20)--(5,10)--(0,0);
  \vertex{0}{40}\vertex{0}{30}\vertex{-5}{15}\vertex{5}{20}\vertex{5}{10}\vertex{0}{0}
}

\lat{22}{-62}{L_{6,4}}{
  \draw (0,20)--(-6,30)--(0,40)--(6,30)--(0,20);
  \draw (0,0)--(0,10) -- (0,20);
  \vertex{0}{0}\vertex{0}{10}\vertex{0}{20}\vertex{-6}{30}\vertex{6}{30}\vertex{0}{40}
}

\lat{44}{-62}{L_{6,4}^d}{
  \draw (0,0)--(-6,10)--(0,20)--(6,10)--(0,0);
  \draw (0,20)--(0,40);
  \vertex{0}{0}\vertex{-6}{10}\vertex{6}{10}\vertex{0}{20}\vertex{0}{30}\vertex{0}{40}
}

\lat{66}{-62}{K_{6,1}}{
  \draw (-8,10)--(0,20) -- (8,10);
  \draw (-3,10)--(0,20);
  \draw (-3,10)--(0,0);
  \draw (3,10)--(0,20);
  \draw (3,10)--(0,0);
  \draw (-8,10)--(0,0);
  \draw (8,10)--(0,0);
  \vertex{0}{20}
  \vertexL{-8}{10}{u}{-11}{10}\vertexL{-3}{10}{v}{-0.5}{10}\vertex{3}{10}\vertexL{8}{10}{w}{11}{10}
  \vertex{0}{0}
}

\lat{88}{-62}{K_{6,2}}{
  \draw (0,0)--(-6,15)--(0,30)--(6,15)--(0,0);
  \draw (0,0)--(0,10)--(0,20)--(0,30);
  \vertex{0}{30}
  \vertexL{-6}{15}{u}{-9}{15}\vertexL{0}{10}{v}{2.5}{12}\vertex{0}{20}\vertexL{6}{15}{w}{9}{15}
  \vertex{0}{0}
}

\end{tikzpicture}
}
\caption{Non-isomorphic lattices with $6$ elements (Part II)}
\label{Fig:Lat_6b}
\end{figure}


\begin{figure}[ht]
\centering
\scalebox{1.0}{
\begin{tikzpicture}[x=0.9mm,y=0.9mm, line width=0.35pt,
  lab/.style={draw=none, fill=none, below=3mm, font=\small}]

\newcommand{\vertex}[2]{\filldraw[black] (#1,#2) circle (1.2mm);}
\newcommand{\vertexL}[5]{\filldraw[black] (#1,#2) circle (1.2mm); \node at (#4,#5) {$#3$};}
\newcommand{\lat}[4]{%
  \begin{scope}[shift={(#1,#2)}]
    #4
    \node[lab] at (0,-3) {$#3$};
  \end{scope}
}


\lat{0}{-124}{K_{6,3}}{
  \draw (0,0)--(-4,10)--(-4,20)--(0,30);
  \draw (0,0)--(4,10)--(4,20)--(0,30);
  \vertex{0}{0}\vertex{-4}{10}\vertex{-4}{20}\vertex{4}{10}\vertex{4}{20}\vertex{0}{30}
}

\lat{22}{-124}{K_{6,4}}{
  \draw (0,0)--(-4,10)--(-4,20)--(0,30);
  \draw (0,0)--(4,10)--(4,20)--(0,30);
  \draw (-4,20)--(4,10);
  \vertex{0}{0}\vertex{-4}{10}\vertex{-4}{20}\vertex{4}{10}\vertex{4}{20}\vertex{0}{30}
}

\lat{44}{-124}{K_{6,5}}{
  \draw (0,0)--(-4,20)--(0,40);
  \draw (0,0)--(4,10)--(4,20)--(4,30)--(0,40);
  \vertex{0}{40}
  \vertexL{-4}{20}{u}{-7}{20}\vertexL{4}{10}{v}{7}{10}\vertex{4}{20}\vertexL{4}{30}{w}{7.5}{30}
  \vertex{0}{0}
}

\lat{66}{-124}{K_{6,6}}{
  \draw (0,10)--(-6,20)--(0,30)--(6,20)--(0,10);
  \draw (0,0)--(0,10);
  \draw (0,30)--(0,40);
  \vertex{0}{0}\vertex{0}{10}\vertex{-6}{20}\vertex{6}{20}\vertex{0}{30}\vertex{0}{40}
}

\lat{88}{-124}{K_{6,7}}{
  \draw (0,0)--(0,8)--(0,16)--(0,24)--(0,32)--(0,40);
  \vertex{0}{0}\vertex{0}{8}\vertex{0}{16}\vertex{0}{24}\vertex{0}{32}\vertex{0}{40}
}

\end{tikzpicture}
}
\caption{Non-isomorphic lattices with $6$ elements (Part III)}
\label{Fig:Lat_6c}
\end{figure}
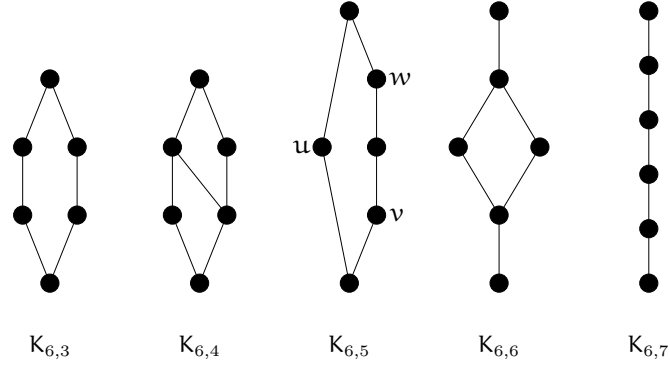

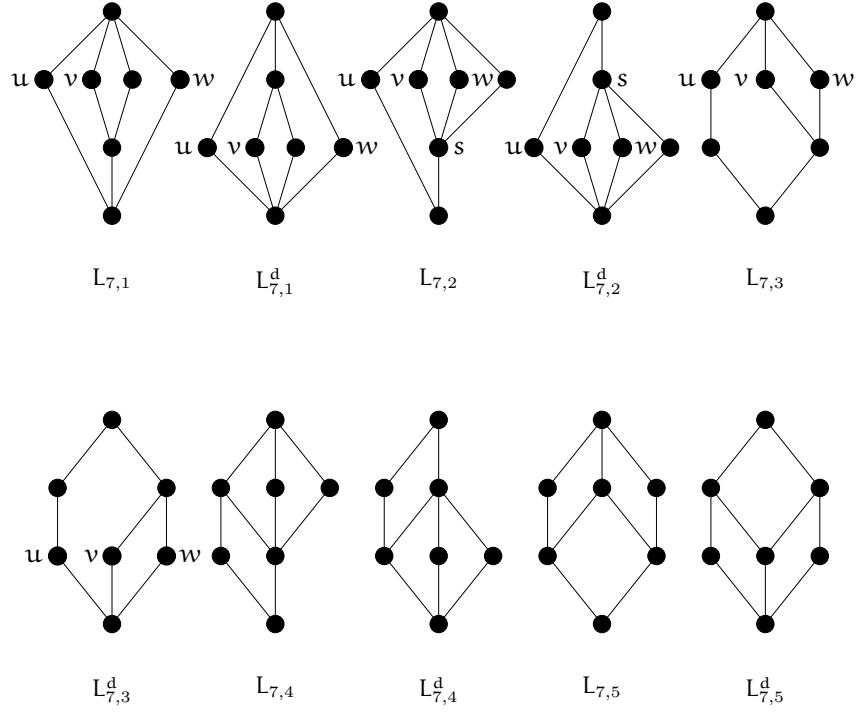
\begin{figure}[ht]
\centering
\scalebox{1.0}{
\begin{tikzpicture}[
  x=0.9mm,y=0.9mm,
  line width=0.35pt,
  lab/.style={draw=none, fill=none, below=3mm, font=\small}]

\newcommand{\vertex}[2]{\filldraw[black] (#1,#2) circle (1.15mm);}
\newcommand{\vertexL}[5]{\filldraw[black] (#1,#2) circle (1.2mm); \node at (#4,#5) {$#3$};}
\newcommand{\lat}[4]{%
  \begin{scope}[shift={(#1,#2)}]
    #4
    \node[lab] at (0,-3) {$#3$};
  \end{scope}
}


\lat{0}{0}{L_{7,1}}{
  \draw (0,30)--(-10,20)--(0,0)--(10,20)--(0,30);
  \draw (0,30)--(-3,20)--(0,10)--(0,0);
  \draw (0,30)--(3,20)--(0,10);
  \vertex{0}{30}\vertexL{-10}{20}{u}{-13.5}{20}\vertexL{10}{20}{w}{13.5}{20}\vertexL{-3}{20}{v}{-6}{20}\vertex{3}{20}\vertex{0}{10}\vertex{0}{0}
}

\lat{24}{0}{L_{7,1}^d}{
  \draw (0,0)--(-10,10)--(0,30)--(10,10)--(0,0);
  \draw (0,0)--(-3,10)--(0,20)--(0,30);
  \draw (0,0)--(3,10)--(0,20);
  \vertex{0}{0}\vertexL{-10}{10}{u}{-13.5}{10}\vertexL{10}{10}{w}{13.5}{10}\vertexL{-3}{10}{v}{-6}{10}\vertex{3}{10}\vertex{0}{20}\vertex{0}{30}
}

\lat{48}{0}{L_{7,2}}{
  \draw (0,30)--(-10,20)-- (0,0); 
  \draw (0,10)--(10,20)--(0,30);
  \draw (0,30)--(-3,20)--(0,10);
  \draw (0,30)--(3,20)--(0,10);
  \draw (0,10)--(0,0);
  \vertex{0}{30}
  \vertexL{-10}{20}{u}{-13.5}{20}
  \vertexL{-3}{20}{v}{-6}{20}\vertexL{3}{20}{w}{6.5}{20}\vertex{10}{20}
  \vertexL{0}{10}{s}{3}{10}\vertex{0}{0}
}

\lat{72}{0}{L_{7,2}^d}{
  \draw (0,0)--(-10,10)-- (0,30); 
  \draw (0,20)--(10,10)--(0,0);
  \draw (0,0)--(-3,10)--(0,20);
  \draw (0,0)--(3,10)--(0,20);
  \draw (0,20)--(0,30);
  \vertex{0}{0}
  \vertexL{-10}{10}{u}{-13}{10}\vertexL{-3}{10}{v}{-6}{10}\vertexL{3}{10}{w}{6.5}{10}
  \vertex{10}{10}\vertexL{0}{20}{s}{3}{20}\vertex{0}{30}
}

\lat{96}{0}{L_{7,3}}{
  \draw (0,30)--(-8,20)--(-8,10)--(0,0)--(8,10)--(8,20)--(0,30);
  \draw (0,30)--(0,20) -- (8,10);
  \vertex{0}{30}\vertex{-8}{10}\vertex{0}{0}\vertex{8}{10}
  \vertexL{-8}{20}{u}{-11.5}{20}
  \vertexL{0}{20}{v}{-3.5}{20}
  \vertexL{8}{20}{w}{11.5}{20}
}


\lat{0}{-60}{L_{7,3}^d}{
  \draw (0,0)--(-8,10)--(-8,20)--(0,30)--(8,20)--(8,10)--(0,0);
  \draw (0,0)--(0,10)--(8,20);
  \vertex{0}{0}\vertexL{-8}{10}{u}{-11.5}{10}
  \vertex{-8}{20}\vertex{0}{30}\vertex{8}{20}
  \vertexL{0}{10}{v}{-3}{10}\vertexL{8}{10}{w}{11.5}{10}
}

\lat{24}{-60}{L_{7,4}}{
  \draw (0,30)--(-8,20)--(-8,10)--(0,0); 
  \draw (8,20)--(0,30);
  \draw (0,0)--(0,30);
  \draw (0,10)--(8,20); \draw (0,10)--(-8,20);
  \vertex{0}{30}\vertex{-8}{20}\vertex{-8}{10}\vertex{0}{20}\vertex{8}{20}\vertex{0}{10}\vertex{0}{0}
}

\lat{48}{-60}{L_{7,4}^d}{
  \draw (0,0)--(-8,10)--(0,20)--(8,10)--(0,0);
  \draw (-8,10)--(-8,20)--(0,30);
  \draw (0,0)--(0,30);
  \vertex{0}{0}\vertex{-8}{10}\vertex{-8}{20}\vertex{0}{10}\vertex{8}{10}\vertex{0}{20}\vertex{0}{30}
}

\lat{72}{-60}{L_{7,5}}{
  \draw (0,30)--(-8,20)--(-8,10)--(0,0)--(8,10)--(8,20)--(0,30);
  \draw (0,20)--(-8,10);
  \draw (0,20)--(8,10);
  \draw (0,20)--(0,30);
  \vertex{0}{30}\vertex{-8}{20}\vertex{-8}{10}\vertex{0}{20}\vertex{0}{0}\vertex{8}{10}\vertex{8}{20}
}

\lat{96}{-60}{L_{7,5}^d}{
  \draw (0,0)--(-8,10)--(-8,20)--(0,30)--(8,20)--(8,10)--(0,0);
  \draw (0,10)--(-8,20);
  \draw (0,10)--(8,20);
  \draw (0,0)--(0,10);
  \vertex{0}{0}\vertex{-8}{10}\vertex{-8}{20}\vertex{0}{10}\vertex{0}{30}\vertex{8}{20}\vertex{8}{10}
}

\end{tikzpicture}
}
\caption{Non-isomorphic lattices with $7$ elements (Part I)}
\label{Fig:Lat7_part1}
\end{figure}
\begin{figure}[ht]
\centering
\scalebox{1.0}{
\begin{tikzpicture}[x=0.9mm,y=0.9mm,
  line width=0.35pt,
  lab/.style={draw=none, fill=none, below=3mm, font=\small}]

\newcommand{\vertex}[2]{\filldraw[black] (#1,#2) circle (1.15mm);}
\newcommand{\vertexL}[5]{\filldraw[black] (#1,#2) circle (1.2mm); \node at (#4,#5) {$#3$};}
\newcommand{\lat}[4]{%
  \begin{scope}[shift={(#1,#2)}]
    #4
    \node[lab] at (0,-3) {$#3$};
  \end{scope}
}


\lat{0}{0}{L_{7,6}}{
  \draw (0,30)--(-10,20)--(0,10)--(10,20)--(0,30);
  \draw (0,30)--(-3,20)--(0,10);
  \draw (0,30)--(3,20)--(0,10);
  \draw (0,10)--(0,0);
  \vertex{0}{30}
  \vertex{-10}{20}\vertex{-3}{20}\vertex{3}{20}\vertex{10}{20}
  \vertexL{0}{10}{u}{-3}{10}\vertex{0}{0}
}

\lat{24}{0}{L_{7,6}^d}{
  \draw (0,0)--(-10,10)--(0,20)--(10,10)--(0,0);
  \draw (0,0)--(-3,10)--(0,20);
  \draw (0,0)--(3,10)--(0,20);
  \draw (0,30)--(0,20);
  \vertex{0}{0}
  \vertexL{-10}{10}{u}{-13}{10}
  \vertexL{-3}{10}{v}{-6}{10}
  \vertexL{3}{10}{w}{6.5}{10}
  \vertex{10}{10}
  \vertexL{0}{20}{s}{-3}{20}\vertex{0}{30}
}

\lat{48}{0}{L_{7,7}}{
  \draw (0,40)--(-8,20)--(0,0);
  \draw (0,40)--(-3,25)--(0,10);
  \draw (0,40)--(8,30)--(8,20)--(0,10);
  \draw (0,10)--(0,0);
  \vertexL{-8}{20}{u}{-11}{20}
  \vertex{0}{40}\vertex{-3}{25}
  \vertexL{8}{30}{w}{11.5}{30}
  \vertex{8}{20}
  \vertexL{0}{10}{v}{3}{10}\vertex{0}{0}
}

\lat{72}{0}{L_{7,7}^d}{
  \draw (0,0)--(-8,20)--(0,40);
  \draw (0,0)--(-3,15)--(0,30);
  \draw (0,0)--(8,10)--(8,20)--(0,30);
  \draw (0,30)--(0,40);
  \vertexL{-8}{20}{u}{-11}{20}
  \vertex{0}{0}\vertex{-3}{15}\vertexL{8}{10}{v}{11}{10}\vertex{8}{20}
  \vertexL{0}{30}{w}{3.5}{30}\vertex{0}{40}
}

\lat{96}{0}{L_{7,8}}{
  \draw (-8,20)--(0,40)--(0,0)--(-8,20);
  \draw (0,40)--(0,30)--(0,20)--(0,10)--(0,0);
  \draw (0,20)--(6,30)--(0,40);
  \vertex{0}{40}
  \vertexL{-8}{20}{u}{-11}{20}
  \vertexL{0}{30}{w}{-2.5}{27}
  \vertex{6}{30}\vertex{0}{20}\vertexL{0}{10}{v}{3}{10}\vertex{0}{0}
}


\lat{0}{-68}{L_{7,8}^d}{
  \draw (-8,20)--(0,0)--(0,40)--(-8,20);
  \draw (0,0)--(0,10)--(0,20)--(0,30)--(0,40);
  \draw (0,20)--(6,10)--(0,0);
  \vertexL{-8}{20}{u}{-11.5}{20}
  \vertex{0}{0}
  \vertex{0}{10}
  \vertexL{6}{10}{v}{9}{10}
  \vertex{0}{20}
  \vertexL{0}{30}{w}{3.5}{30.5}
  \vertex{0}{40}
}

\lat{24}{-68}{L_{7,9}}{
  \draw (0,0)--(-6,10)--(-6,30)--(0,40);
  \draw (0,0)--(6,10)--(6,30)--(6,30)--(0,40);
  \draw (-6,10)--(6,30);
  \vertex{0}{0}\vertex{-6}{10}\vertex{-6}{30}\vertex{0}{40}
  \vertex{6}{10}\vertex{6}{20}\vertex{6}{30}
}

\lat{48}{-68}{L_{7,9}^d}{
  \draw (0,40)--(-6,30)--(-6,10)--(0,0);
  \draw (0,40)--(6,30)--(6,20)--(6,10)--(0,0);
  \draw (-6,30)--(6,10);
  \vertex{0}{40}\vertex{-6}{30}\vertex{-6}{10}\vertex{0}{0}
  \vertex{6}{30}\vertex{6}{20}\vertex{6}{10}
}

\lat{72}{-68}{L_{7,10}}{
  \draw (0,0)--(-6,10)--(-6,30)--(0,40);
  \draw (0,0)--(6,10)--(6,20)--(6,30)--(0,40);
  \draw (-6,10)--(6,20);
  \vertex{0}{0}\vertex{-6}{10}\vertex{-6}{30}\vertex{0}{40}
  \vertex{6}{10}\vertex{6}{20}\vertex{6}{30}
}

\lat{96}{-68}{L_{7,10}^d}{
  \draw (0,40)--(-6,30)--(-6,10)--(0,0);
  \draw (0,40)--(6,30)--(6,20)--(6,10)--(0,0);
  \draw (-6,30)--(6,20);
  \vertex{0}{40}\vertex{-6}{30}\vertex{-6}{10}\vertex{0}{0}
  \vertex{6}{30}\vertex{6}{20}\vertex{6}{10}
}
\end{tikzpicture}
}
\caption{Non-isomorphic lattices with $7$ elements (Part II)}
\label{Fig:Lat7_part2}
\end{figure}
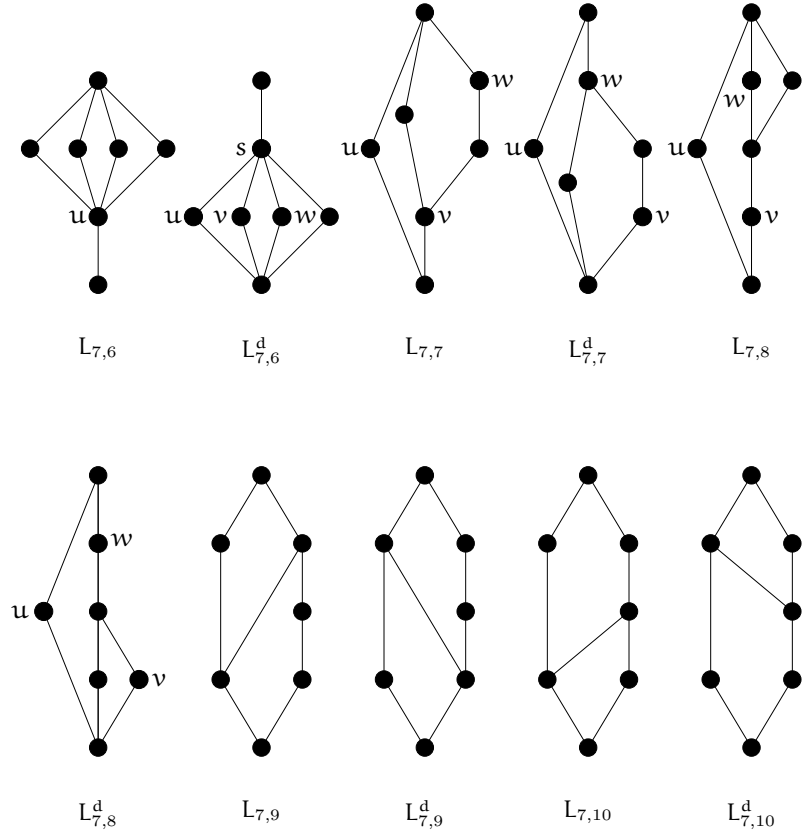
\begin{figure}[ht]
\centering
\scalebox{1.0}{
\begin{tikzpicture}[x=0.9mm,y=0.9mm,
  line width=0.35pt,
  lab/.style={draw=none, fill=none, below=3mm, font=\small}]

\newcommand{\vertex}[2]{\filldraw[black] (#1,#2) circle (1.15mm);}
\newcommand{\vertexL}[5]{\filldraw[black] (#1,#2) circle (1.2mm); \node at (#4,#5) {$#3$};}
\newcommand{\lat}[4]{%
  \begin{scope}[shift={(#1,#2)}]
    #4
    \node[lab] at (0,-3) {$#3$};
  \end{scope}
}


\lat{0}{0}{L_{7,11}}{
  \draw (0,40)--(-10,25)--(0,10)--(10,25)--(0,40);
  \draw (0,10)--(0,0);
  \draw (0,40)--(0,30)--(0,20)--(0,10);
  \vertex{0}{40}\vertex{-10}{25}\vertex{10}{25}\vertex{0}{30}\vertex{0}{20}
  \vertexL{0}{10}{u}{3}{10}\vertex{0}{0}
}

\lat{24}{0}{L_{7,11}^d}{
  \draw (0,0)--(-10,15)--(0,30)--(10,15)--(0,0);
  \draw (0,30)--(0,40);
  \draw (0,0)--(0,10)--(0,20)--(0,30);
  \vertex{0}{0}\vertexL{-10}{15}{v}{-13}{15}\vertexL{10}{15}{w}{13.5}{15}\vertexL{0}{10}{u}{3}{10}\vertex{0}{20}\vertexL{0}{30}{s}{3}{30}\vertex{0}{40}
}

\lat{48}{0}{L_{7,12}}{
  \draw (0,40)--(-10,25)--(0,10);
  \draw (0,40)--(0,30)--(0,20)--(0,10)--(0,0);
  \draw (0,30)--(10,20)--(0,10);
  \vertex{0}{40}\vertex{-10}{25}\vertex{0}{30}\vertex{0}{20}\vertex{10}{20}\vertex{0}{10}\vertex{0}{0}
}

\lat{72}{0}{L_{7,12}^d}{
  \draw (0,0)--(-10,15)--(0,30);
  \draw (0,40)--(0,30)--(0,20)--(0,10)--(0,0);
  \draw (0,10)--(10,20)--(0,30);
  \vertex{0}{0}
  \vertexL{-10}{15}{u}{-13}{15}\vertex{0}{10}
  \vertexL{0}{20}{v}{-3}{20}
  \vertexL{10}{20}{w}{13.5}{20}\vertex{0}{30}
  \vertex{0}{40}
}

\lat{96}{0}{L_{7,13}}{
  \draw (0,40)--(-10,25)--(0,10);
  \draw (0,40)--(0,30)--(0,20)--(0,10)--(0,0);
  \draw (0,40)--(10,30)--(0,20);
  \vertex{0}{40}\vertex{-10}{25}\vertex{0}{30}\vertex{10}{30}\vertex{0}{20}
  \vertexL{0}{10}{u}{3}{10}\vertex{0}{0}
}


\lat{0}{-68}{L_{7,13}^d}{
  \draw (0,0)--(-10,15)--(0,30);
  \draw (0,40)--(0,30)--(0,20)--(0,10)--(0,0);
  \draw (0,20)--(10,10)--(0,0);
  \vertex{0}{0}\vertex{-10}{15}\vertex{0}{10}\vertex{10}{10}\vertex{0}{20}\vertex{0}{30}\vertex{0}{40}
}

\lat{24}{-68}{L_{7,14}}{
  \draw (0,40)--(-10,30)--(-10,20)--(0,10);
  \draw (0,40)--(10,30)--(10,20)--(0,10);
  \draw (0,10)--(0,0);
  \vertex{0}{40}\vertex{-10}{30}\vertex{-10}{20}\vertex{10}{30}\vertex{10}{20}\vertex{0}{10}\vertex{0}{0}
}

\lat{48}{-68}{L_{7,14}^d}{
  \draw (0,0)--(-10,10)--(-10,20)--(0,30);
  \draw (0,0)--(10,10)--(10,20)--(0,30);
  \draw (0,30)--(0,40);
  \vertex{0}{0}\vertex{-10}{10}\vertex{-10}{20}\vertex{10}{10}\vertex{10}{20}\vertex{0}{30}\vertex{0}{40}
}

\lat{72}{-68}{L_{7,15}}{
  \draw (0,40)--(-10,30)--(-10,20)--(0,10);
  \draw (0,40)--(10,30)--(10,20)--(0,10);
  \draw (0,10)--(0,0);
  \draw (-10,30)--(10,20);
  \vertex{0}{40}\vertex{-10}{30}\vertex{-10}{20}\vertex{10}{30}\vertex{10}{20}\vertex{0}{10}\vertex{0}{0}
}

\lat{96}{-68}{L_{7,15}^d}{
  \draw (0,0)--(-10,10)--(-10,20)--(0,30);
  \draw (0,0)--(10,10)--(10,20)--(0,30);
  \draw (0,30)--(0,40);
  \draw (-10,10)--(10,20);
  \vertex{0}{0}\vertex{-10}{10}\vertex{-10}{20}\vertex{10}{10}\vertex{10}{20}\vertex{0}{30}\vertex{0}{40}
}

\end{tikzpicture}
}
\caption{Non-isomorphic lattices with $7$ elements (Part III)}
\label{Fig:Lat7_part3}
\end{figure}
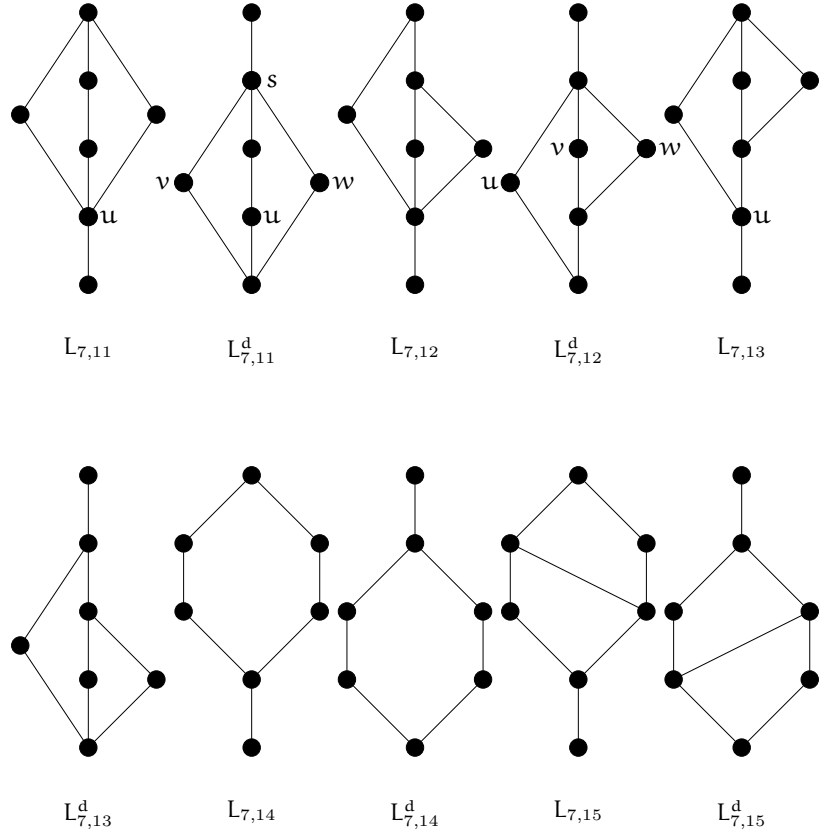
\begin{figure}[ht]
\centering
\scalebox{1.0}{
\begin{tikzpicture}[x=0.9mm,y=0.9mm,
  line width=0.35pt,
  lab/.style={draw=none, fill=none, below=3mm, font=\small}]

\newcommand{\vertex}[2]{\filldraw[black] (#1,#2) circle (1.15mm);}
\newcommand{\vertexL}[5]{\filldraw[black] (#1,#2) circle (1.2mm); \node at (#4,#5) {$#3$};}
\newcommand{\lat}[4]{%
  \begin{scope}[shift={(#1,#2)}]
    #4
    \node[lab] at (0,-3) {$#3$};
  \end{scope}
}


\lat{0}{0}{L_{7,16}}{
  \draw (0,0)--(0,10)--(0,20);
  \draw (0,20)--(-10,30);
  \draw (0,20)--(0,30);
  \draw (0,20)--(10,30);
  \draw (-10,30)--(0,40);
  \draw (0,30)--(0,40);
  \draw (10,30)--(0,40);
  \vertex{0}{0}\vertex{0}{10}\vertex{0}{20}\vertex{-10}{30}\vertex{0}{30}\vertex{10}{30}\vertex{0}{40}
}

\lat{24}{0}{L_{7,16}^d}{
  \draw (0,40)--(0,30)--(0,20);
  \draw (0,20)--(-10,10);
  \draw (0,20)--(0,10);
  \draw (0,20)--(10,10);
  \draw (-10,10)--(0,0);
  \draw (0,10)--(0,0);
  \draw (10,10)--(0,0);
  \vertex{0}{40}\vertex{0}{30}\vertex{0}{20}\vertex{-10}{10}\vertex{0}{10}\vertex{10}{10}\vertex{0}{0}
}

\lat{48}{0}{L_{7,17}}{
  \draw (0,0)--(0,10);
  \draw (0,10)--(-5,30)--(0,50);
  \draw (0,10)--(5,20)--(5,30)--(5,40)--(0,50);
  \vertex{0}{0}\vertexL{0}{10}{u}{3}{10}\vertex{-5}{30}\vertex{5}{20}\vertex{5}{30}\vertex{5}{40}
  \vertex{0}{50}
}

\lat{72}{0}{L_{7,17}^d}{
  \draw (0,50)--(0,40);
  \draw (0,40)--(-5,20)--(0,0);
  \draw (0,40)--(5,30)--(5,20)--(5,10)--(0,0);
  \vertex{0}{50}\vertex{0}{40}
  \vertexL{-5}{20}{u}{-8}{20}\vertexL{5}{30}{w}{8.5}{30}\vertex{5}{20}\vertexL{5}{10}{v}{8}{10}\vertex{0}{0}
}

\lat{96}{0}{L_{7,18}}{
  \draw (0,0)--(0,10)--(0,20);
  \draw (0,20)--(-6,35)--(0,50);
  \draw (0,20)--(6,30)--(6,40)--(0,50);
  \vertex{0}{0}\vertex{0}{10}\vertex{0}{20}\vertex{-6}{35}\vertex{6}{30}\vertex{6}{40}\vertex{0}{50}
}


\lat{0}{-72}{L_{7,18}^d}{
  \draw (0,50)--(0,40)--(0,30);
  \draw (0,30)--(-6,15)--(0,0);
  \draw (0,30)--(6,20)--(6,10)--(0,0);
  \vertex{0}{50}\vertex{0}{40}\vertex{0}{30}\vertex{-6}{15}\vertex{6}{20}\vertex{6}{10}\vertex{0}{0}
}

\lat{24}{-72}{L_{7,19}}{
  \draw (0,0)--(0,10)--(0,20);
  \draw (0,20)--(-10,30)--(0,40)--(0,50);
  \draw (0,20)--(10,30)--(0,40);
  \vertex{0}{0}\vertex{0}{10}\vertex{0}{20}\vertex{-10}{30}\vertex{10}{30}\vertex{0}{40}\vertex{0}{50}
}

\lat{48}{-72}{L_{7,19}^d}{
  \draw (0,50)--(0,40)--(0,30);
  \draw (0,30)--(-10,20)--(0,10)--(0,0);
  \draw (0,30)--(10,20)--(0,10);
  \vertex{0}{50}\vertex{0}{40}\vertex{0}{30}\vertex{-10}{20}\vertex{10}{20}\vertex{0}{10}\vertex{0}{0}
}

\lat{72}{-72}{L_{7,20}}{
  \draw (0,0)--(0,10)--(0,20)--(0,30);
  \draw (0,30)--(-10,40)--(0,50);
  \draw (0,30)--(10,40)--(0,50);
  \vertex{0}{0}\vertex{0}{10}\vertex{0}{20}\vertex{0}{30}\vertex{-10}{40}\vertex{10}{40}\vertex{0}{50}
}

\lat{96}{-72}{L_{7,20}^d}{
  \draw (0,50)--(0,40)--(0,30)--(0,20);
  \draw (0,20)--(-10,10)--(0,0);
  \draw (0,20)--(10,10)--(0,0);
  \vertex{0}{50}\vertex{0}{40}\vertex{0}{30}\vertex{0}{20}\vertex{-10}{10}\vertex{10}{10}\vertex{0}{0}
}

\end{tikzpicture}
}
\caption{Non-isomorphic lattices with $7$ elements (Part IV)}
\label{Fig:Lat7_part4}
\end{figure}
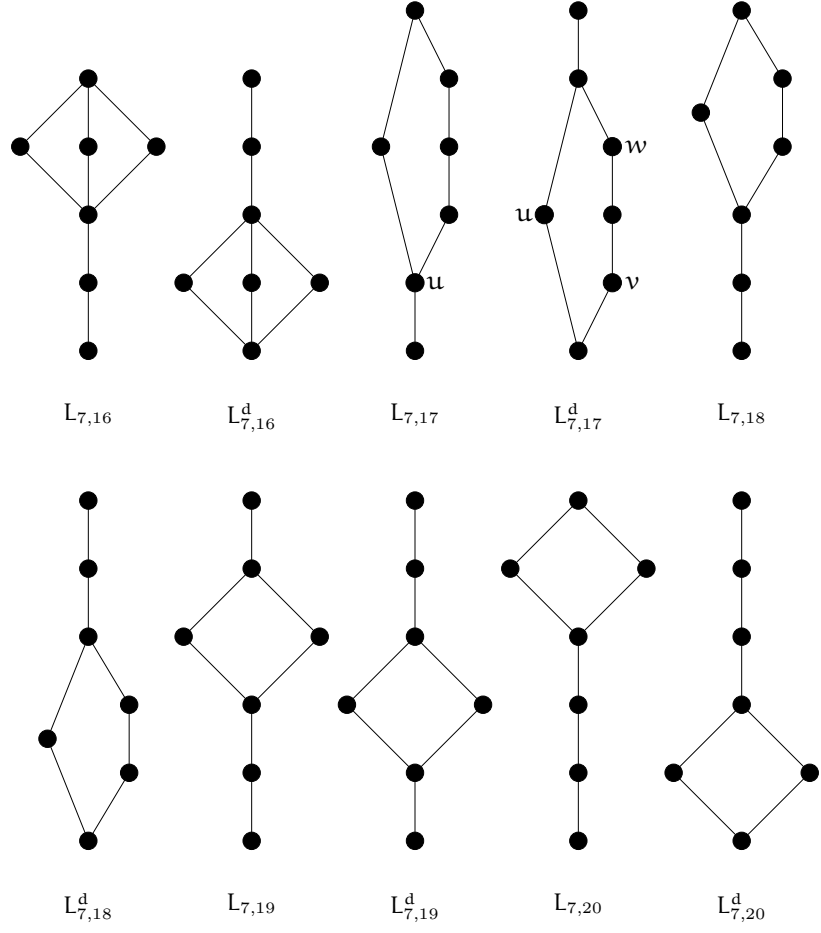
\begin{figure}[ht]
\centering
\scalebox{1.0}{
\begin{tikzpicture}[x=0.9mm,y=0.9mm,
  line width=0.35pt,
  lab/.style={draw=none, fill=none, below=3mm, font=\small}]

\newcommand{\vertex}[2]{\filldraw[black] (#1,#2) circle (1.15mm);}
\newcommand{\vertexL}[5]{\filldraw[black] (#1,#2) circle (1.2mm); \node at (#4,#5) {$#3$};}
\newcommand{\lat}[4]{%
  \begin{scope}[shift={(#1,#2)}]
    #4
    \node[lab] at (0,-3) {$#3$};
  \end{scope}
}


\lat{0}{0}{K_{7,1}}{
  \draw (0,30)--(-12,15)--(0,0)--(12,15)--(0,30);
  \draw (0,30)--(-6,15)--(0,0);
  \draw (0,30)--(0,15)--(0,0);
  \draw (0,30)--(6,15)--(0,0);
  \vertex{0}{30}\vertexL{-12}{15}{u}{-15}{15}\vertexL{-6}{15}{v}{-9}{15}\vertex{0}{15}\vertex{6}{15}\vertexL{12}{15}{w}{9}{15}\vertex{0}{0}
}

\lat{30}{0}{K_{7,2}}{
  \draw (0,0)--(-9,15)--(0,30);
  \draw (0,0)--(-3,15)--(0,30);
  \draw (0,0)--(3,15)--(0,30);
  \draw (0,0)--(9,10)--(9,20)--(0,30);
  \vertex{0}{0}
  \vertexL{-9}{15}{u}{-12}{15}\vertexL{-3}{15}{v}{-6}{15}\vertexL{3}{15}{w}{6.5}{15}\vertex{9}{10}\vertex{9}{20}\vertex{0}{30}
}

\lat{60}{0}{K_{7,3}}{
  \draw (0,0)--(-8,10)--(-8,20)--(0,30);
  \draw (0,0)--(0,15)--(0,30);
  \draw (0,0)--(8,10)--(8,20)--(0,30);
  \vertex{0}{0}\vertexL{-8}{10}{u}{-11}{10}\vertex{-8}{20}\vertexL{0}{15}{v}{-3}{15}\vertexL{8}{10}{w}{11.5}{10}\vertex{8}{20}\vertex{0}{30}
}

\lat{90}{0}{K_{7,4}}{
  \draw (0,0)--(-8,10)--(-8,20)--(0,30);
  \draw (0,0)--(8,10)--(8,20)--(0,30);
  \draw (-8,10)--(8,20);
  \draw (0,0)--(16,15)--(0,30);
  \vertex{0}{0}\vertex{-8}{10}\vertexL{-8}{20}{u}{-11}{20}\vertexL{8}{10}{v}{5}{10}\vertex{8}{20}\vertex{0}{30}\vertexL{16}{15}{w}{19.5}{15}
}


\lat{0}{-70}{K_{7,5}}{
  \draw (0,40)--(-8,20)--(0,0);
  \draw (0,40)--(8,20)--(0,0);
  \draw (0,40)--(0,30)--(0,20)--(0,10)--(0,0);
  \vertex{0}{40}\vertexL{-8}{20}{u}{-11}{20}\vertex{0}{30}\vertexL{0}{20}{v}{-3}{20}
  \vertex{0}{10}\vertexL{8}{20}{w}{11.5}{20}\vertex{0}{0}
}

\lat{25}{-70}{K_{7,6}}{
  \draw (0,40)--(0,30)--(0,20)--(0,10)--(0,0);
  \draw (-8,20)--(0,30);
  \draw (-8,20)--(0,10);
  \draw (8,20)--(0,40);
  \draw (8,20)--(0,0);
  \vertex{0}{40}\vertex{0}{30}\vertex{-8}{20}\vertex{0}{20}\vertex{8}{20}\vertex{0}{10}\vertex{0}{0}
}

\lat{50}{-70}{K_{7,7}}{
  \draw (0,40)--(-6,30)--(-6,10)--(0,0);
  \draw (0,40)--(6,30)--(6,20)--(6,10)--(0,0);
  \vertex{0}{40}\vertex{-6}{30}\vertexL{-6}{10}{u}{-9}{10}
  \vertexL{6}{30}{w}{9.5}{30}
  \vertex{6}{20}
  \vertexL{6}{10}{v}{9}{10}\vertex{0}{0}
}

\lat{75}{-70}{K_{7,8}}{
  \draw (0,40)--(-6,30)--(-6,10)--(0,0);
  \draw (0,40)--(6,30)--(6,20)--(6,10)--(0,0);
  \draw (-6,10)--(6,30);
  \vertex{0}{40}
  \vertexL{-6}{30}{w}{-9.5}{30}\vertex{-6}{10}\vertex{6}{30}
  \vertexL{0}{20}{v}{-3}{20}
  \vertexL{6}{10}{u}{9}{10}\vertex{0}{0}
}

\lat{100}{-70}{K_{7,9}}{
  \draw (0,0)--(-6,10)--(0,20)--(6,10)--(0,0);
  \draw (0,20)--(-6,30)--(0,40)--(6,30)--(0,20);
  \vertex{0}{0}\vertex{-6}{10}\vertex{6}{10}\vertex{0}{20}\vertex{-6}{30}\vertex{6}{30}\vertex{0}{40}
}


\lat{0}{-140}{K_{7,10}}{
  \draw (0,40)--(0,30)--(0,20)--(0,10)--(0,0);
  \draw (-8,20)--(0,30);
  \draw (-8,20)--(0,10);
  \draw (8,20)--(0,30);
  \draw (8,20)--(0,10);
  \vertex{0}{40}\vertex{0}{30}\vertex{-8}{20}\vertex{0}{20}\vertex{8}{20}\vertex{0}{10}\vertex{0}{0}
}

\lat{30}{-140}{K_{7,11}}{
  \draw (0,0)--(-6,25)--(0,50);
  \draw (0,0)--(6,10)--(6,25)--(6,30)--(6,40)--(0,50);
  \vertex{0}{0}\vertexL{-6}{25}{u}{-9}{25}\vertex{0}{50}
  \vertexL{6}{10}{v}{9}{10}\vertex{6}{20}\vertexL{6}{30}{w}{9.5}{30}\vertex{6}{40}
}

\lat{60}{-140}{K_{7,12}}{
  \draw (0,0)--(0,10);
  \draw (0,10)--(-6,25)--(0,40);
  \draw (0,10)--(6,20)--(6,30)--(0,40);
  \draw (0,40)--(0,50);
  \vertex{0}{0}\vertex{0}{10}\vertex{-6}{25}\vertex{6}{20}\vertex{6}{30}\vertex{0}{40}\vertex{0}{50}
}

\lat{90}{-140}{K_{7,13}}{
  \draw (0,48)--(0,40)--(0,32)--(0,24)--(0,16)--(0,8)--(0,0);
  \vertex{0}{48}\vertex{0}{40}\vertex{0}{32}\vertex{0}{24}
  \vertex{0}{16}\vertex{0}{8}\vertex{0}{0}
}

\end{tikzpicture}
}
\caption{Non-isomorphic self-dual lattices with $7$ elements}
\label{Fig:Lat7_part5}
\end{figure}
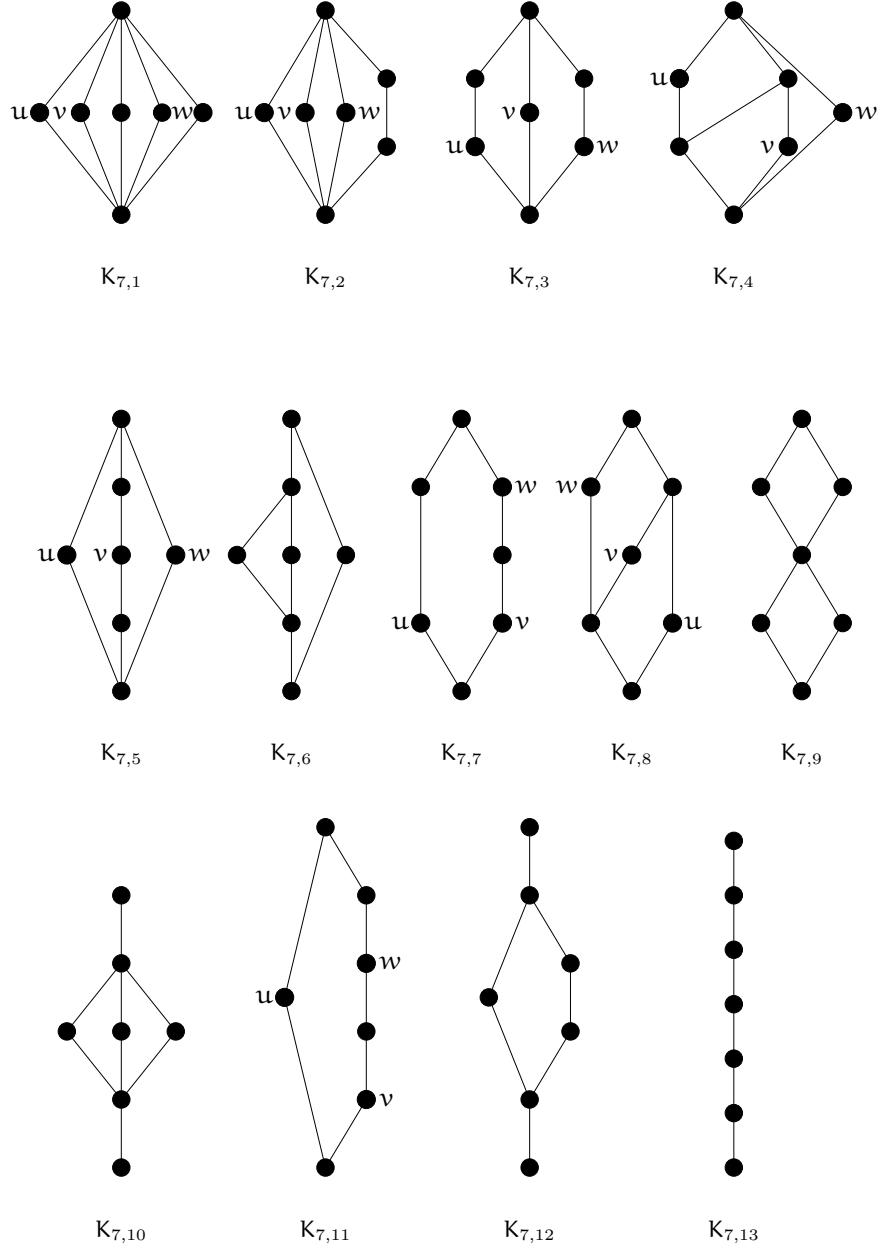 


\end{document}